\documentclass[11pt,reqno,a4paper]{amsart}
\usepackage{verbatim,rcs,cite}
\usepackage{amssymb,amsfonts,latexsym,mathrsfs}
\usepackage{amsmath}
\usepackage{amsthm, pstricks}
\usepackage{mathrsfs}
\usepackage[all]{xy}
\usepackage{longtable}
\usepackage{pstricks}

\newtheorem{theorem}{Theorem}[section]
\newtheorem{lemma}[theorem]{Lemma}
\newtheorem{proposition}[theorem]{Proposition}
\newtheorem{corollary}[theorem]{Corollary}

\theoremstyle{remark}

\newtheorem{remark}[theorem]{Remark}
\theoremstyle{definition}
\newtheorem{definition}[theorem]{Definition}

\DeclareMathOperator{\core}{core}   \DeclareMathOperator{\Sym}{Sym}
 \DeclareMathOperator{\Alt}{Alt}\DeclareMathOperator{\Fix}{Fix}

\newcommand{\C}{\mathscr{C}}

\RCS $Revision:  $ \RCSdate $Date: 2008/01/29 13:25:27 $

\title[Smallest Degree for a Direct Product Obeying an Inequality]{The Smallest Faithful Permutation Degree for a Direct Product obeying an Inequality Condition}

\date{}

\author{David Easdown}
\address{School of Mathematics and Statistics, University of Sydney, NSW 2006, Australia}
\email{david.easdown@sydney.edu.au}

\author{Neil Saunders}
\address{Heilbronn Institute for Mathematical Research, University of Bristol, School of Mathematics, University Walk, Bristol BS8 1TW, United Kingdom}
\email{neil.saunders@bristol.ac.uk}

\begin{document}


\maketitle

\begin{abstract}
The minimal faithful permutation degree $\mu(G)$ of a finite group $G$ is the least nonnegative integer $n$ such that $G$ embeds
in the symmetric group $\Sym(n)$. 
Clearly $\mu(G \times H) \le \mu(G) + \mu(H)$ for all finite groups $G$ and $H$. Wright (1975) proves that equality occurs when $G$ and $H$ are nilpotent 
and exhibits an example of strict inequality where $G\times H$
embeds in $\Sym(15)$. 
 Saunders (2010) produces an infinite family of examples of permutation groups $G$ and $H$ where
$\mu(G \times H) < \mu(G) + \mu(H)$, including the example of Wright's as a special
case. The smallest groups in Saunders' class embed in $\Sym(10)$. 
In this paper we prove that 10 is minimal in the sense that
$\mu(G \times H) = \mu(G) + \mu(H)$ for all groups $G$ and $H$
such that $\mu(G\times H)\le 9$.
\end{abstract}

\section{Introduction}
Throughout this paper all groups are assumed to be finite. The {\it minimal faithful permutation degree} $\mu(G)$ of a group $G$ is the smallest nonnegative integer such that $G$ embeds in the symmetric group $\Sym(n)$. 
Recall that the {\it core} of a subgroup $H$ of $G$, denoted by
$\core(H)$, is the largest normal subgroup of $G$ contained in $H$, and that $H$ is {\it core-free} if $\core(H)$ is trivial. Thus $\mu(G)$ is the
smallest sum of indexes for a collection of subgroups $G_1,\ldots,G_\ell$ of $G$ such $\cap_{i=1}^{\ell} G_i$ is core-free.
The subgroups $G_i$ are the respective point-stabilisers for the action of $G$ on its orbits and letters in the $i$th orbit may be identified with cosets of $G_i$.
If $\ell = 1$ then the representation is transitive and $G_1$ is a core-free subgroup.  

For any groups $G$ and $H$, we always have the inequality
\begin{equation} \label{eq:directsum}
\mu(G \times H) \leq \mu(G) + \mu(H).
\end{equation}
Johnson and Wright (see \cite{J71, W75}) developed a general theory of minimal degrees of groups and 
described conditions for when equality occurs in \eqref{eq:directsum}. They proved this to be the case 
when $G$ and $H$ have coprime orders and when $G$ and $H$ are nilpotent.
Easdown and Praeger (see \cite{EP88}) showed that equality holds when $G$ and $H$ are direct products of simple groups. 
Wright in \cite{W75} asked whether equality occurs in  \eqref{eq:directsum} always and an example 
exhibiting strict inequality was attached as an addendum, where $G$ and $H$ are given as subgroups of $\Sym(15)$. In that example, $G$ and $H$ generate a subgroup
$GH$ of $\Sym(15)$ that is an internal direct product of $G$ and $H$.

Saunders showed in \cite{S08} that the example in [8] fits into a general family that provides infinitely many instances of
strict inequality in \eqref{eq:directsum}.  There $G$ could be taken to be the complex reflection group $G(p,p,q)$, 
where $p$ and $q$ are distinct odd primes satisfying certain other conditions, 
and $H$ the centraliser of the minimally embedded image of $G$ in $\Sym(pq)$. In this family, it was always the case that  $$\mu(G(p,p,q))=\mu\left(G(p,p,q)\times
C_{\Sym(pq)}(G(p,p,q))\right)=pq,$$ and so examples of strict inequality in \eqref{eq:directsum} was assured. 
The smallest minimal degree of a direct product in this family was $10$, furnished by taking $G$ to be $G(2,2,5)$ and 
thus $H=C_{\Sym(10)}(G) \cong C_2$. In fact, one can take $G$ to be a split extension of the product of $4$ copies of $C_2$ 
(a so-called {\it deleted permutation module} for $\Sym(5)$ over $\mathbb{F}_2$) by any subgroup of $\Sym(5)$ 
that contains the $5$-cycle (see \cite{S08} for a description of these complex reflection groups and exposition of the examples).

The main result of this paper (Theorem \ref{theorem:degree9} below) is that if $G$ and $H$ are groups such that $G\times H$ embeds in $\Sym(9)$ then 
equality occurs in \eqref{eq:directsum}.
Thus, to find groups $G$ and $H$ such that $G\times H$ embeds in $\Sym(n)$ and strict inequality occurs in \eqref{eq:directsum}, one requires $n\ge 10$.

\section{Background and Preliminaries} \label{section:background}

Wright in \cite{W75} considered the class $\mathscr{C}$ of groups $G$ such that
$\mu(G)=\mu(G_1)$ for some nilpotent subgroup $G_1$ of $G$.
Wright noted (see Claim $1$ and Claim $2$ of \cite{W75}) that
all symmetric, alternating and dihedral groups are members of $\mathscr{C}$.
Because equality occurs in \eqref{eq:directsum} whenever $G$ and $H$ are nilpotent (see Theorem 2 of \cite{W75}), the following lemma is immediate and used often below without comment.
\begin{lemma}\label{lemma:wright}
If $G, H\in\mathscr{C}$ then $G \times H \in \mathscr{C}$ and $\mu(G \times H)=\mu(G) + \mu(H)$.
\end{lemma}

We now briefly state some background results that we will need during the course of our later proofs. Here, we follow the notation of \cite{DM96},
where further exposition and complete proofs can be found.

\begin{definition}\label{definition:block}
Let $G$ be a subgroup of $\Sym(A)$.
\begin{enumerate}
\item[(i)] We say that $G$ acts {\it semi-regularly} if, for all $x \in A$, $x^g=x$ implies $g=1$.
\item[(ii)] We say $G$ acts {\it regularly} if $G$ acts transitively and semi-regularly. 
\item[(iii)] If $G$ acts transitively then a {\it block} for $G$ is a subset $B$ of $A$ such that, for all $g \in G$, $B^g \cap B = \varnothing$ or $B^g=B$.
\item[(iv)] If $H$ is a subgroup of $G$ then the {\it set of fixed points of $H$ in $A$} is
$$\Fix(H)=\{x \in A\,|\,x^h=x,\,\text{for all} \, \,h \in H \}.$$
\end{enumerate}
\end{definition}

\begin{theorem} \label{theorem:centraliser}
Let $G$ be a transitive subgroup of \,$\Sym(A)$ and $H$ the stabiliser of a point in $A$. Then $C:=C_{\Sym(A)}(G) \cong N_{G}(H)/H$ and $C$
acts semi-regularly on $A$. \end{theorem}

\begin{corollary} \label{corollary:centraliser1}
Let $G$ be a subgroup of \,$\Sym(A)$ where $A=A_1 \cup \ldots \cup A_k$ and the $A_i$ are the orbits of $G$, all of different sizes.
Then
$$C_{\Sym(A)}(G) \cong N_{G}(H_1)/H_1 \times \ldots \times N_{G}(H_k)/H_k$$ where $H_i$ is the stabiliser of a point in $A_i$ for
$i=1,\ldots,k$.
\end{corollary}

\begin{proof}
Without loss of generality, we may suppose that  $A_i=\{H_{i}g\,|\,g
\in G \}$ for each $i$.
The map that takes $x\in N_{G}(H_i)$ to the permutation $H_{i}g \mapsto H_{i}x^{-1}g$ for $g\in G$ induces an isomorphism
$$\Phi_{i}:N_{G}(H_i)/H_i \longrightarrow C_i$$ where $C_{i}:=C_{\Sym(A_i)}(G_{|A_i})$. Gluing these maps together, we get an embedding
$$\Phi=(\Phi_1, \ldots ,\Phi_k): N_{G}(H_1)/H_1 \times \ldots \times N_{G}(H_k)/H_k \longrightarrow C_{\Sym(A)}(G),$$ where the images are
juxtaposed in the usual way since the orbits are disjoint. We claim that $\Phi$ is onto. \vspace{0pt}

Let $\theta$ be an arbitrary element of $C_{\Sym(A)}(G)$.  Suppose first that, for some $i\neq j$, 
there exists an $x \in A_i$ such that $x\theta \in A_j$.
Since $\theta$ centralises $G$, the restriction of $\theta$ to $A_i$ is an injective map into $A_j$, so that $|A_i| \leq |A_j|$.
But $\theta^{-1}$ also centralises $G$, so similarly $|A_j| \leq
|A_i|$, whence $|A_i|=|A_j|$, contradicting that the orbits have different sizes. Hence the orbits of $\theta$ respect the partition of $A$ given by $A_1, \ldots,A_k$. For each $i=1, \ldots, k,$ $\theta_{|A_{i}}: H_i \mapsto H_{i}g_{i},$ for some $g_{i} \in G$ and it quickly follows that $g_i \in N_{G}(H_i)$
and $\theta_{|A_{i}}=(H_{i}g_{i}^{-1})\Phi_{i}$. Hence $\theta=(H_{1}g_{1}^{-1},\ldots,H_{k}g_{k}^{-1})\Phi$, completing the proof that $\Phi$ is onto.
\end{proof}

The following propositions are well-known (see \cite{DM96} or \cite{C99} for example). 

\begin{proposition} \label{proposition:fix}
Let $G$ be a transitive subgroup of \,$\Sym(A)$ and $H$ the stabiliser of a point. Then $\Fix(H)$ is a block for $G$, the induced permutation
group on the block $\Fix(H)$ is regular and $|\Fix(H)|=|N_{G}(H)/H|$.
\end{proposition}

\begin{proposition} \label{proposition:block}
Let $G$ be a transitive subgroup of\, $\Sym(A)$ with non-trivial block $B$. Let $G_{B}^{B}$ denote the induced permutation group on the block
$B$ and $\bar{G}$ denote the induced action on the set of blocks. Then $G$ embeds in the wreath product $G_{B}^{B} \wr \bar{G}$.
\end{proposition}

We list a few more technical observations here of a general nature relating to minimal embeddings which we will use repeatedly later on.

\begin{lemma} \label{lemma:L1}
Suppose $G$ is a subgroup of \,$\Sym(A_1) \times \ldots \times \Sym(A_k)$ such that $\mu(G)=|A_1|+ \ldots +|A_k|$.
\begin{enumerate}
\item[(i)] $\mu(G\pi)=\sum_{i \in X}|A_i|$ for any projection $\pi$ onto $\prod_{i \in X}\Sym(A_i)$ for $X$ any subset of $\{1,\ldots, k\}$. 
\item[(ii)] For all $i$ in $\{1,\ldots,k\}$, there exists an $\alpha \neq 1$ such that $(1, \ldots,\alpha,\ldots ,1)$ is contained in $G$, where
$\alpha$ is located in the $i$-th place. 
\item[(iii)] If $|A_1|=2$, then $G \cong C_2 \times H$ where $H$ is a subgroup of\, $\Sym(A_2) \times \ldots \times \Sym(A_k)$ and
$\mu(H)=|A_2|+\ldots+|A_k|$. 
\item[(iv)] If $|A_1|=3$, then $(\alpha,1,\ldots,1)$ is an element of $G$ for some $3$-cycle $\alpha$.
\item[(v)] If $|A_1|=4$ and $(\alpha, 1,\ldots ,1)$ is an element of $G$ for some $3$-cycle $\alpha$, then $G$ contains $\Alt(A_1)\times \{1\}\times\ldots\times\{1\}$.
\item[(vi)] If $|A_1|=4$, say $A_1=\{a,b,c,d\}$, and $((a\,b), 1,\ldots,1)\in G$, then $((c\,d), 1,\ldots,1)\in G$.
\item[(vii)] If $G\pi$ is transitive where $\pi$ is the projection onto the first coordinate and $|A_1|=p$ for some prime $p$ such that $p>|A_i|$
for all $i \in \{2,\ldots,k\}$, then $(\alpha,1,\ldots,1)$ is an element of $G$ for some $p$-cycle $\alpha$.
\end{enumerate}
\end{lemma}

\begin{proof} Put $n=|A_1|+\ldots+|A_k|$.
For (i), we observe that if $\mu(G\pi) < \sum_{i \in X}|A_i|$, then pasting projections gives an embedding of $G$ in $\Sym(n-1)$, contradicting that $\mu(G)=n$. 

For (ii), let $\pi$ be the projection onto $\prod_{j\neq i} \Sym(A_j)$. Then $\ker(\pi_{|G})$ is non-trivial, for otherwise $G$ embeds inside
$\Sym(n-1)$, again contradicting that $\mu(G)=n$.

For (iii), suppose $|A_1|=2$. By (ii), $g=(\alpha,1,\ldots,1) \in G$, where $\Sym(A_1)=\{1, \alpha\}$. Let $K$ be the kernel of the projection
of $G$ onto the first coordinate, so $$K \cap \langle g \rangle = \{(1,\ldots,1)\} \quad \text{and} \quad G=\langle g \rangle K,$$ so that $G$
is the internal direct product of $\langle g \rangle$ and $K$. But the first coordinate of each element of $K$ is $1$, so $K \cong H$, where $H$
is the result of ignoring the first coordinate. Thus $G \cong \langle g \rangle \times K \cong C_2 \times H.$ Clearly $H$ is a subgroup of
$\Sym(A_2) \times \ldots \times \Sym(A_k)$ and $H=G\pi$ where $\pi$ projects onto $\Sym(A_2) \times \ldots \times \Sym(A_k)$, so that
$\mu(H)=|A_2| + \ldots + |A_k|,$ by (i). 

For (iv), suppose $|A_1|=3$. By (ii), we have that $g=(\alpha,1,\ldots,1)\in G$ for some $\alpha \neq 1$ in $\Sym(A_1)$. If
$\alpha$ is a $3$-cycle, then we are done so suppose $\alpha=(a\,b)$ where $A_1=\{a,b,c\}$. Since $\mu(G)=n$ there exists
$h=(\beta,h_2, \ldots, h_k)\in G$ for some $\beta$ that moves $c$, that is, $\beta$ is an element of 
$\{(a\,c), (b\,c), (a\,b\,c), (a\,c\,b) \}$. It follows that $[\alpha, \beta]=(a\,b\,c)$ or  $(a\,c\,b)$ and so $[g,h]=([\alpha,\beta],1\ldots,1)$ is
contained in $G$.

For (v) and (vi), suppose $|A_1|=4 $, say $A_1=\{a,b,c,d\}$. If $g=(\alpha,1,\ldots,1)\in G$ where $\alpha = (a\,b\,c)$ then,
since $\mu(G)=n$, there exists some $h=(\beta, 1, \ldots, 1)\in G$ such that $\beta $ moves $d$, and then
$\langle g, g^{h} \rangle \cong \Alt(A_1)\times \{1\}\times\ldots\times\{1\}$, verifying (v).
For (vi), suppose that $(\alpha,1,\ldots,1)\in G$ where $\alpha = (a\,b)$.
Since $\mu(G)=n$ there exists some $(\beta, 1, \ldots, 1), (\gamma, 1, \ldots, 1) \in G$ such that $\beta $ moves $c$ and $\gamma$ moves $d$. 
It follows quickly that $(c\,d)\in \langle \alpha, \beta, \gamma\rangle$ so that $((c\,d), 1,\ldots,1)\in G$.

For (vii), suppose $G\pi$ is transitive where $\pi$ is projection onto the first coordinate and $A_1=p$, where $p$ is a prime and $p
>|A_i|$ for all $i \geq 2$. Then $G$ contains an element $(\alpha_1, \ldots, \alpha_k)$ of order $p$, since $p$ divides $|G|$. This implies that
$\alpha_2=\ldots=\alpha_k=1$ since there is insufficient room for $p$-cycles in $\Sym(A_2), \ldots, \Sym(A_k)$.
\end{proof}

\begin{corollary} \label{corollary:centraliser2}
Let $G$ be a subgroup of \,$\Sym(n)$ such that $\mu(G)=n$ and $n\le 9$.
Then
$C:=C_{\Sym(n)}(G)$ is abelian. In particular $C\in \C$. 
\end{corollary}

\begin{proof}
The conclusion follows from Theorem \ref{theorem:centraliser}, Corollary \ref{corollary:centraliser1} and Lemma \ref{lemma:L1}, by considering
all partitions of $n$ and noting that $N_G(H)/H$ must have order $1$, $2$, $3$, $4$ or $5$, whenever $H$ is the stabiliser of a letter in any given orbit.
\end{proof}

The following simple observation will be used repeatedly in the sequel.

\begin{proposition} \label{proposition:centraliser}
Let $G \in \mathscr{C}$ with $\mu(G)=n$ and identify $G$ with its embedded image in $\Sym(n)$. Let $C:=C_{\Sym(n)}(G)$ be the centraliser of $G$
in $\Sym(n)$ with respect to this minimal embedding and suppose that $C$ is nilpotent. Then every nontrivial subgroup of $C$
intersects $G$ nontrivially.
\end{proposition}

\begin{proof}
Suppose for a contradiction that there is a nontrivial subgroup $P$ of $C$ such that $G \cap P=\{1\}$. Then $\langle G, P \rangle=GP$ is a subgroup of
$\Sym(n)$ that is an
internal direct product of $G$ and $P$. But $P$ is nilpotent, being a subgroup of a nilpotent group, so $P \in
\mathscr{C}$. Since $G \in \mathscr{C}$ we have $$\mu(GP)=\mu(G \times P)=\mu(G)+\mu(P) > n,$$ contradicting that $\mu(GP)\le n$.
\end{proof}

\begin{corollary} \label{corollary:centraliser}
With notation as in the previous proposition, if $C$ is elementary abelian, then $C \leq G$.
\end{corollary}

\begin{proof}
If $C$ is elementary abelian and $1\not=c\in C$ then $\langle c\rangle$ is a subgroup of order $p$, so $c\in G$ by the previous proposition.
\end{proof}

\section{Case by Case that $10$ is Minimal} \label{section:proof}

In this section we will prove Theorem \ref{theorem:degree9} below in stages, so that there are no examples of strict inequality 
in \eqref{eq:directsum} in $\Sym(n)$ for $n \leq 9$. 
This is clear for $n=2$ and $n=3$.
Our approach for $n=4, \ldots, 9$ is to
show that, for a minimally embedded group $G$ in $\Sym(n)$, 
there is no nontrivial subgroup of the centraliser of $G$ in $\Sym(n)$ that intersects trivially with $G$.  
For the most part, this will follow by applying Proposition \ref{proposition:centraliser},
revealing the pervasiveness of Wright's class $\mathscr{C}$ for permutation groups of small degree.

\subsection{ The $\Sym(4)$, $\Sym(5)$ and $\Sym(6)$ Cases}

\begin{proposition} \label{proposition:class456}
Let $G$ be a finite group such that $\mu(G) \leq 6$. Then $G \in \mathscr{C}$.
\end{proposition}

\begin{proof}
First suppose that $\mu(G)=4$. 
If $G$ acts intransitively with respect to the embedding in $\Sym(4)$, then $G\cong C_2 \times C_2$ by Lemma \ref{lemma:L1} (iii), so
$G\in \mathscr{C}$.
Suppose that $G$ acts transitively. Then $G$ has a core-free subgroup $H$
of index $4$, so that the Sylow $2$-subgroups of $G$ have size $4$ or $8$. Hence, a copy of $C_4$ or $C_2 \times C_2$ is a subgroup of
$G$. Both of these are nilpotent and have minimal degree $4$, so again $G\in\mathscr{C}$. 

Now suppose that $\mu(G)=5$. If $G$ acts transitively, then $G$ contains a subgroup of index $5$ and so
contains a copy of $C_5$, implying that $G\in\mathscr{C}$. Suppose that $G$ acts intransitively. By minimality, the action of $G$ must have two orbits, of sizes
$2$ and $3$ respectively.
By Lemma \ref{lemma:L1} (iii),(iv), $G$ contains a subgroup isomorphic to $C_3 \times C_2$, so again $G \in \mathscr{C}$.

Finally suppose that $\mu(G)=6$. We may identify $G$ with its embedded image in $\Sym(6)$. If $9$ or $16$ divides $|G|$, 
then $G$ contains a Sylow $2$ or $3$-subgroup of $\Sym(6)$ and hence a copy of $C_3 \times C_3$ or $C_2 \times C_2 \times C_2$,
so that $G \in \C$. 

We may suppose therefore that neither $9$ or $16$ divides $|G|$. 
Suppose first that $G$ acts intransitively. It follows by Lemma \ref{lemma:L1} (iii), (iv), (v), (vi) that
$G$ contains a subgroup that is an internal direct product of subgroups $K$ and $L$ of minimal degrees less that $6$ but adding up to $6$.
By previous cases, $K, L\in \mathscr{C}$, so $K\times L\in \mathscr{C}$, and it follows that $G\in \mathscr{C}$. 

Henceforth we may suppose that $G$ acts transitively. In particular, $G$ contains a subgroup of index $6$, 
so that the possible orders of $G$ are $6, 12, 24, 30, 60$ or $120$.
If $|G|=6$ or $12$ then, from the Appendix, $\mu(G)\not=6$, which is impossible.
If $|G|=24$ then, from the Appendix, 
 either $\mu(G)\not=6$, which is impossible, or $G$ contains a copy of $C_2 \times C_2 \times C_2$, 
so that $G \in \C$. If $|G|=30$ then $G$ contains a copy of $C_5 \times C_3$ 
so that $\mu(G)\ge 8$, which is again impossible. 

Henceforth we may suppose that $|G|=60$ or $120$.
Suppose that $G$ has a nontrivial proper normal subgroup $N$ of order less that $60$. 
If $|N|=2$, $4$ or $8$, then $N$ and any Sylow $5$-subgroup together generate a subgroup of $G$ containing an abelian subgroup of order $10$, 
so that $\mu(G) \geq 7$, a contradiction. If a Sylow $3$-subgroup or Sylow $5$-subgroup is normal in $G$, 
then it follows that $G$ contains a copy of $C_5 \times C_3$ and so $\mu(G) \geq 8$, which is also impossible. 
These observations force $|N|$ to be $12$ or $24$ and for $N$ to contain a non-normal Sylow $3$-subgroup. 
But then there are exactly four Sylow $3$-subgroups of $G$, and the kernel of the conjugation action on them must be a normal subgroup of $G$, forced to have order $12$ or $24$ containing a normal Sylow $3$-subgroup, which is a contradiction.

Hence $G$ has no proper normal subgroup of order less than $60$, 
so that $G$ is isomorphic to $\Alt(5)$ or $\Sym(5)$ and so $\mu(G)=5$, which is impossible. This completes the proof of the proposition.
\end{proof}

\subsection{The $\Sym(7)$ Case}

Throughout this subsection, put
\begin{equation}\label{H}
H:=\langle (1\,2\,3), (1\,2)(4\,5\,6\,7) \rangle
\cong\langle a,b\,|\,a^3=b^4=1,\,a^b=a^{-1}\rangle.
\end{equation}

\begin{proposition} \label{proposition:subdirect}
Let $H$ be as in \eqref{H}. Then $\mu(H)=7$, $H \not\in \mathscr{C}$ and $H$ is up to isomorphism the unique proper subdirect product of $\Sym(3) \times C_4$. Further 
$C_{\Sym(7)}(H)=\langle (4\,5\,6\,7)\rangle.$
\end{proposition}

\begin{proof}
It is easily verified that $\mu(H)=7$ (and minimal degrees of groups of order 12 are listed in the Appendix) and that the nilpotent subgroups of $G$ are isomorphic to $C_2, C_3, C_4$ and $C_2 \times C_3$, all of which have minimal
degree strictly less than $7$. Hence $G \not\in \mathscr{C}$.
It is easy to check that $\Sym(3) \times C_4$ has a unique subgroup of order 12 containing an element of order 4, which must therefore be isomorphic to $H$.

Put $z=(4\,5\,6\,7)$. Clearly $\langle z \rangle \subseteq C:=C_{\Sym(7)}(H)$. Note that the orbits of $H$ are $\{1, 2,3\}$ and $\{4, 5, 6, 7\}$ and are of different sizes. By Corollary \ref{corollary:centraliser1}, 
$C \cong N_{H}(H_3)/H_3 \times N_{H}(H_4)/H_4$, where 
$H_3=\langle (1\,2)(4\,5\,6\,7)\rangle$ and
$H_4=\langle (1\,2\,3)\rangle$ are the stabilisers of $3$ and $4$ respectively.
But $N_{H}(H_3)=H_3$
and $N_{H}(H_4)=H$, so that $C \cong H/H_4 \cong C_4$. Therefore $C=\langle z \rangle$.
\end{proof}

This group $H$ is also unique in the following sense.

\begin{theorem} \label{theorem:G7}
Let $G$ be a group such that $\mu(G)=7$ and $G \not\in \mathscr{C}$. Then the image of any minimal embedding of $G$ in $\Sym(7)$ is permutation
equivalent to $H$. In particular, $G\cong H$.
\end{theorem}

\begin{proof}
We may regard $G$ as a subgroup of $\Sym(7)$. If $G$ is transitive then $7$ divides $|G|$, being the index of a point stabiliser, so that $G$ contains a copy of $C_7$, and $\mu(G)=7=\mu(C_7)$, contradicting that $G \not\in \mathscr{C}$. Hence $G$
is intransitive. 

If $G$ has an orbit of size $2$, then by Lemma \ref{lemma:L1} (iii), $G \cong C_2 \times K$ for some group $K$ such that $\mu(K)=5$, so $K \in
\mathscr{C}$, by Proposition \ref{proposition:class456}, whence $G \in \mathscr{C}$, a contradiction. It follows that $G$ has one orbit of size $3$ and one of size $4$.
Without loss of generality we may suppose these orbits are $X_1=\{1, 2, 3\}$ and $X_2=\{4, 5, 6, 7\}$. By parts (ii) and (iv) of Lemma \ref{lemma:L1}, we may, without loss of generality, assume that $\alpha:= (1\,2\,3) \in G$ and there exists $\beta \in G$ such that $\beta$ fixes $X_1$ pointwise and moves a letter from $X_2$.

If $\beta$ is a 4-cycle then $G$ contains the subgroup $\langle \alpha, \beta \rangle \cong C_3 \times C_4$. If $\beta$ is a 3-cycle or a 2-cycle, then, by Lemma \ref{lemma:L1} (v) and (vi), $G$ contains a subgroup isomorpic to $C_3\times \Alt(4)$ or $C_3\times C_2\times C_2$ respectively. In each of these cases, $G\in \C$, leading to a contradiction.

Hence $\beta$ must be a product of two disjoint 2-cycles.
Without loss of generality, we suppose that $\beta=(4\,6)(5\,7).$
For any $\gamma \in G$, we will write
$\gamma_1=\gamma|_{X_1}$ and $\gamma_2=\gamma|_{X_2},$ so that $\gamma=\gamma_{1}\gamma_{2}$. 

Let $\pi$ be projection onto $\Sym(X_2)$, so that $G\pi$ must be a transitive subgroup of $\Sym(X_2)$. 
By \cite[Table 2.1]{DM96} the transitive subgroups of $\Sym(X_2)$ are itself, $\Alt(X_2)$ or isomorphic to $C_2 \times C_2$, $D_8$ or $C_4$.   \newline

\noindent \underline {\textbf{Case (i):}}\, $G\pi = \Sym(X_2)$ or $\Alt(X_2)$. \newline
There is some $\gamma\in G$ such that $\gamma_2=(4\, 5 \, 6)$, with $\gamma_1$ being a $2$-cycle or a power of $\alpha$. 
It readily follows that $\gamma_2 \in G$. By Lemma \ref{lemma:L1} (v), $C_3 \times \Alt(4)$ is isomorphic to a subgroup of $G$, and it follows that $G \in \mathscr{C}$, a contradiction. \newline 

\noindent \underline {\textbf{Case (ii):}}\, $G\pi =\langle (4\,5)(6\,7), (4\,6)(5\,7)\rangle$. \newline
There is some $\gamma\in G$ such that $\gamma_2=(4\, 5)(6 \, 7)$. If $\gamma_1\in \langle \alpha\rangle$ then $\gamma_2\in G$ and $\langle \alpha, \beta, \gamma_2\rangle$ is a subgroup of $G$ isomorphic to $C_3\times C_2\times C_2$, so that $G\in \C$, a contradiction.  Hence $\gamma_1\not\in \langle \alpha\rangle$, so, without loss of generality, $\gamma_1=(1\,2)$. Then
$
\langle \alpha, \beta,\gamma\rangle
$
is a subgroup of $G$ isomorphic to $\Sym(3)\times C_2$, of minimal degree $5$, so cannot exhaust all of $G$. Hence there exists some $\delta\in G\backslash \langle \alpha, \beta,\gamma\rangle$.
If $\delta_2=1$ or $\beta$, then $\delta_1\in G$ and $\delta_1$ is a 2-cycle, so that
$\langle \alpha, \beta, \gamma, \delta\rangle\cong \Sym(3)\times C_2\times C_2$, so that $G\in \C $, a contradiction.
Hence $\delta_2=(4\,5)(6\,7)$ or $(4\,7)(5\,6)$ and 
$\delta_1\in \langle \alpha\rangle$. Then $\langle \alpha, \beta, \delta\rangle\cong C_3\times C_2\times C_2$, so that $G\in \C $, a contradiction. 
\newline

\noindent \underline {\textbf{Case (iii):}}\, $G\pi \cong D_8$. \newline
Either $\beta$ is central or $\beta$ inverts a 4-cycle in $G\pi$.
Suppose first that $\beta$ is central in $G\pi$. Then there are some
$\gamma,\delta\in G$ such that $\gamma_2=(4\,5\,6\,7)$ and
$\delta_2=(4\,6)$.
If $\gamma_1\in\langle\alpha\rangle$ or $\delta_1\in\langle\alpha\rangle$ then $\langle \alpha, \gamma\rangle\cong C_3\times C_4$ or
$\langle \alpha, \beta,\delta\rangle\cong C_3\times C_2\times C_2$ , so that $G\in\C$, a contradiction.
Hence $\gamma_1$ and $\delta_1$ are both 2-cycles.
By conjugating $\delta$ by a power of $\alpha$, without any loss of generality, we may assume
$\gamma_1=\delta_1$. But then $\gamma\delta=(4\,5)(6\,7)$, so that
$\langle \alpha, \beta,\gamma\delta\rangle\cong C_3\times C_2\times C_2$, so that $G\in\C$, a contradiction.
Hence $\beta$ inverts a 4-cycle in $G\pi$, so there is some $\varepsilon\in G$ such that $\varepsilon_2=(4\,6\,5\,7)$.
If $\varepsilon_1\in\langle\alpha\rangle$ then $\langle \alpha, \varepsilon\rangle\cong C_3\times C_4$, so that $G\in\C$, a contradiction. Hence $\varepsilon_1$ is a 2-cycle, so that
$\langle \alpha, \beta, \varepsilon^2\rangle\cong C_3\times C_2\times C_2$, so that $G\in \C$, a contradiction.

Cases (i), (ii) and (iii) produce contradictions, so we must have 
$G\pi \cong C_4$. Hence there is some $\gamma\in G$ such that
$\gamma_2=(4\,5\,6\,7)$. If $\gamma_1\in\langle\alpha\rangle$ then
$\langle\alpha,\gamma\rangle\cong C_3\times C_4$, so that $G\in\C$, a contradiction. Hence $\gamma_1$ is a 2-cycle, and it follows that $(1\,2)(4\,5\,6\,7)\in G$, so that $H\le G$. If $H\not=G$ then $G\cong \Sym(3)\times C_4$, so that $G\in\C$, a contradiction. Hence $H=G$, and the theorem is proved.
\end{proof}

\subsection{ The $\Sym(8)$ Case}

In this section we prove that all intransitive subgroups and all but two transitive subgroups of $\Sym(8)$ of minimal degree $8$ are members of $\C$. The two exceptions up to isomorphism that are not members of $\C$ 
(see Theorem \ref{theorem:class8} below)
turn out to be primitive:
\begin{align*}
K:=
\langle
(1\,2)(3\,4)(5\,6)(7\,8),(1\,3)(2\,4)(5\,7)(6\,8),
(1\,5)(2\,6)&(3\,7)(4\,8), (2\, 3\,5\, 4\, 7\,8\, 6)\rangle\\
&\cong (C_2\times C_2\times C_2) \rtimes C_7\,;
\end{align*}
\begin{align*}
L:=
K\langle(3\,5\,7)(4\,6\,8)\rangle
\cong (C_2\times C_2\times C_2) \rtimes (C_7\rtimes C_3)\,.
\end{align*}
Note (for the proof of Theorem \ref{theorem:no8} below) that
both $C_{\Sym(8)}(K)$ and $C_{\Sym(8)}(L)$ are trivial.

\begin{proposition} \label{proposition:class8}
Let $G$ be group such that $\mu(G)=8$ and the minimal faithful representation of $G$ is intransitive. Then $G \in \mathscr{C}$.
\end{proposition}

\begin{proof}
We may suppose throughout that $G$ is not a $2$-group. 
If $G$ has an orbit of size $2$ or $3$ then, by Lemma \ref{lemma:L1}, $G$ contains a subgroup that is an internal direct product of subgroups $K$ and $L$ of minimal degrees less than $7$ but adding up to $8$, so, by Proposition \ref{proposition:class456}, $K$ and $L$ both lie in $\C$, whence $G \in \C$.

Hence we may suppose that $G$ has exactly two orbits of size $4$, which me may take to be $\{1,2,3,4\}$ and $\{5,6,7,8\}$. In particular, $|G|$ must be divisible by $3$ and $4$, but not by $5$ or $7$. 
Let $S$ denote a Sylow $2$-subgroup and $T$ a Sylow $3$-subgroup of $G$.
 If $32$ divides $|S|$, then $\mu(S) \geq 8$ (because $32$ does not divide $7!$), so that $\mu(S)=8$ and $G \in \C$. If $9$ divides $|G|$, then $T$ is a Sylow $3$-subgroup of $\Sym(8)$, so, without loss of generality, $T=\langle (1\,2\,3), (5\,6\,7) \rangle$ and, by Lemma \ref{lemma:L1} (v), $G$ contains the subgroup $\Alt(\{1, 2 ,3, 4\}) \times \Alt(\{5,6,7,8\})$, so $G \in \C$.

Henceforth we may suppose that 32 and 9 
do not divide $|G|$, so that $|G|= 12$, $24$, or $48$.
From the Appendix, the only possibilities for $G$, up to isomorphism, are $SL(2,3)$, $GL(2,3)$, $C_4\times \Alt(4)$ and 
$\Alt(4)\rtimes C_4$.
Then $G$ contains a copy of $Q_8$ in the first two cases,
a copy of $C_4\times C_4$ in the third case, and a copy of $C_2\times C_2\times C_4$ in the last case. In all cases $G$ contains a nilpotent subgroup of minimal degree $8$, so $G\in \C$.
\end{proof}
\begin{theorem} \label{theorem:class8}
 Let $G$ be a group such that $\mu(G)= 8$ and the minimal permutation representation is transitive. Then $G\in  \C$ 
or $G$ is isomorphic to $(C_2\times C_2\times C_2)\rtimes C_7$ or
$(C_2\times C_2\times C_2)\rtimes  (C_7 \rtimes C_3)$, where the semidirect product actions in each case are nontrivial and unique up to isomorphism.
\end{theorem}

\begin{proof}
Again we may suppose that $G$ is not a $2$-group and that $32$ does not divide $|G|$. 
Since the representation is transitive, $G$ contains a core-free subgroup of index $8$. In particular $8$ divides $|G|$. 
The order of $|G|$ must then be one of the following:
$$
24, \;40,\;56, \;72,\;120, \;168, \;280, \;360,\;504,\;840,\;2520,
$$
$$
48, \;80, \;112, \;144, \;240, \;336,\; 560,\;720,\;1008,\; 1680,\;5040.
$$
Let $S$ denote some Sylow 2-subgroup of $G$.
In the first row $|S|=8$, and in the second row $|S|=16$. If $\mu(S)=8$ then $G\in \C$. Hence we may suppose in the following that $\mu(S)< 8$, so that $S$ is isomorphic to $C_2\times C_2\times C_2$, $C_4\times C_2$, $D_8$ or $D_8\times C_2$.

We carefully consider each possibility for $|G|$, either obtaining a contradiction or verifying that $G\in\C$ or $G$ is isomorphic to one of the two groups listed. Note that for each order in the second row, $|G|/8=2k$ where $k$ is odd, so that $G$ also contains a core-free subgroup of order $k$. Thus, in all cases in both rows, we may suppose that $G$ contains a core-free subgroup $H$ of order $k$, where $k$ is the largest odd divisor of $|G|$.

If $|G|=24$ or $48$ then, as in the proof of the previous proposition, $G\in \C$. 
If $|G|=40$ then, from the Appendix, $\mu(G)\not=8$, a contradiction. 
If $|G|=56$ then, from the Appendix, 
$G\cong (C_2\times C_2\times C_2) \rtimes C_7$.
If $|G|=120$, $240$, $360$ or $720$ then $|H|=15$ or $45$ and it follows that $G$ contains a subgroup isomorphic to $C_3\times C_5$ of minimal degree $8$, so that $G\in \C$.
If $|G|=280$, $504$, $560$, $840$, $1008$ or $1680$ then $|H|=35$, $63$ or $105$ so that $G$ contains a subgroup isomorphic to $C_5\times C_7$ or $C_3\times C_7$ of minimal degree at least 10, contradicting that $\mu(G)=8$.

Suppose that $|G|=80$. Then $|H|=5$ and $S\cong D_8\times C_2$.
Because $H$ is core-free, there must be 16 Sylow 5-subgroups, forcing $S$ to be normal in $G$. Hence $G$ is the internal semidirect product of $S$ by $H$. The centre $Z$ of $S$ is isomorphic to $C_2\times C_2$ with no automorphisms of order $5$, so $ZH$ is a an abelian subgroup of $G$ of order $20$ and minimal degree $9$, contradicting that $\mu(G)=8$.

In all remaining cases, $k$ is divisible by $7$ or $9$. If $G$ has a normal subgroup of order $2$ or $4$, then $G$ contains a copy of $C_{14}$, so that $\mu(G)>8$, a contradiction, or $G$ contains a copy of $C_2\times C_3\times C_3$, so that $G\in \C$. 
If $G$ has a normal subgroup $N$ of order $16$ then $N\cong D_8\times C_2$, whose centre $Z$ is a characteristic subgroup of order $4$, so that $Z$ is a normal subgroup of $G$, and we are back in a previous case. Henceforth, we may suppose that $G$ has no normal subgroup of order $2$, $4$ or $16$.

Suppose that $|G|=72$ or $144$. 
Then $|H|=9$, so that 
$H\cong C_3\times C_3$ is a Sylow 3-subgroup of $G$. Since $H$ is core-free, there must be exactly 4 or 16 Sylow 3-subgroups of $G$.
In the first case, let $K$ be
the nontrivial kernel of the conjugation action of $G$ on these 4 Sylow $3$-subgroups.
If $K$ has an element of order $3$ then $H$ is not core-free, a contradiction. 
Thus $K$ must have an element $\alpha$ of order $2$ so that $\langle \alpha, H\rangle\cong C_2\times C_3\times C_3$ and $\mu(\langle \alpha, H\rangle)=8$, whence $G\in \C$.
We may suppose then that $G$ has 16 Sylow $3$-subgroups and $|G|=144$. Because $G$ is not simple, it must have a nontrivial proper normal subgroup $N$. If $3$ divides but $16$ does not divide $|N|$, then $N$, and therefore also $G$, has at most $4$ Sylow $3$-subgroups, contradicting that $G$ has 16 Sylow $3$-subgroups.
If $|N|=8$ then  $G$ contains a subgroup of order 72 and index 2, so normal in $G$, and we are back in the previous case. 
It remains to consider the case $|N|=48$.
But now $N$ must have 16 Sylow 3-subgroups and a normal Sylow 2-subgroup, which is also a normal subgroup of
$G$ of order $16$, a contradiction.

Suppose that $|G|=112$. Then $|H|=7$ and $S\cong D_8\times C_2$. Put $K=N_G(H)$. Because $H$ is core-free, there must be 8 Sylow 7-subgroups of $G$, so $|K|=14$. 
Because $G$ is not simple, it must have a nontrivial proper normal subgroup $N$. If $|N|=7$, 14 or 28 then there is a unique Sylow 7-subgroup of $N$, which must be normal in $G$, contradicting that $H$ is core-free. 
If $|N|=8$ then $NH$ is a subgroup of $G$ of order 56 and index 2, which must be normal. 
We may suppose then that $|N|=56$. Let $T$ be a Sylow 2-subgroup of $N$.
From the Appendix,
$N\cong (C_2\times C_2\times C_2)\rtimes C_7$, 
so $T\cong C_2\times C_2\times C_2$ and
$T$  is normal in both $N$ and $G$, and $N=TH$. Let $z$ be any element of order $2$ in $K$ and put $Z=\langle z\rangle$ so $K=HZ$. Then $G=NZ=THZ=TK$ is an internal semidirect product of $T$ by $K$. If the conjugation action of $K$ on $T$, regarded as a vector space over the field with 3 elements, is faithful then ${\rm GL}(2,3)$ contains a subgroup of order 14,
contradicting the well-known fact that the normaliser of a Sylow 7-subgroup of ${\rm GL}(2,3)$ has order 21. Certainly the action of $H$ on $T$ is faithful (because $H$ is not normal in $N$), so the action of $z$ on $T$ must be trivial. Hence $S\cong TZ\cong C_2\times C_2\times C_2\times C_2$, contradicting that $S\cong D_8\times C_2$.

Suppose $|G|=168$. Then
$|H|=21$ and 
 $S\cong C_2\times C_2\times C_2$, $C_4\times C_2$ or $D_8$. If $H\cong C_3\times C_7$ then $\mu(G)\ge 10$, contradicting that $\mu(G)=8$. It follows that $H\cong C_7\rtimes C_3$, where the number of Sylow $3$-subgroups of $H$ is 7. If $G$ is simple then it is well-known that $G$ has a subgroup of index $7$, so that $\mu(G)\le 7$, a contradiction. Hence $G$ has a nontrivial proper normal subgroup $N$. If the order of $N$ is 3 or 6 then it follows that $G$ contains an element of order 21 so that $\mu(G)\ge 10$, a contradiction.  If the order of $N$ is 7, 14, 21, 28, 42 or 84 then there is a unique Sylow 7-subgroup of $N$, which must be normal in $G$, contradicting that $H$ is core-free. If the order of $N$ is $12$ or $24$ then $N$ must contain all Sylow 3-subgroups of $G$ and there can be at most 4 of them, contradicting that $H$ has 7 Sylow 3-subgroups. If $|N|= 56$ then $G$ is the internal semidirect product of $N$ by a cyclic group of order $3$ and, from the Appendix, $N\cong (C_2\times C_2\times C_2)\rtimes C_7$, so that 
$G\cong (C_2\times C_2\times C_2)\rtimes (C_7 \rtimes C_3)$, containing a normal subgroup of order 8. Thus we may suppose $|N|=8$, 
so that $N=S$. 
Then $G$ is the internal semidirect product of $S$ by $H$. If $S\cong C_4\times C_2$ or $D_8$ then it follows that $G$ contains an element of order 14, so that $\mu(G) \ge 9$, a contradiction. Hence $S\cong C_2\times C_2\times C_2$ so that again
$G\cong (C_2\times C_2\times C_2)\rtimes (C_7 \rtimes C_3)$ and we are done.

Suppose $|G|=336$. Then $|H|=21$ and $S\cong D_8\times C_2$. As in the case $|G|=168$, we have $H\cong C_7\rtimes C_3$ 
with 7 Sylow $3$-subgroups. Because $G$ is not simple, it must have a nontrivial proper normal subgroup $N$. If the order of $N$ is 3, 6, 7, 12, 14, 21, 24, 28, 42 or 84 then we obtain contradictions as in the case $|G|=168$.
If $|N|=48$ then, from the table, either $G\in \C$, and we are done, or $N\cong C_2\times {\rm Sym}(4)$ has a characteristic subgroup of order 24 (isomorphic to $C_2\times {\rm Alt}(4)$), which is then normal in $G$, and we are again back in the earlier list.
If $|N|=112$ then we obtain a contradiction as before (in the case $|G|=112$), with minor adjustments (with any Sylow 7-subgroup in the role of $H$). If $|N|=8$ then $|NH|=168$. If $|N|=56$ then joining $N$ with any Sylow-3 subgroup again gives a subgroup of order 168. In either case we get a subgroup of $G$ of index $2$, so without loss of generality we may suppose $|N|=168$. Choose any element $z$ of order $2$ outside $N$, which must exist because $S\cong D_8\times C_2$, and put $Z=\langle z\rangle$. Then $G=NZ$ is a semidirect product. If the action of $z$ is an inner automorphism, say by conjugation by an element $n$ of $N$, which we may also take to be of order $2$, then it follows that the action of $nz$ is trivial, and we can find an element of $G$ of order 14, again leading to a contradiction. 
Hence we may suppose the action of $z$ is by an outer automorphism.
If $N$ is simple then $N\cong {\rm GL}(3,2)$ and
$G\cong {\rm GL}(3,2)\rtimes C_2$, and it is well known that its Sylow 2-subgroup is isomorphic to $D_{16}$, contradicting that $S\cong D_8\times C_2$. Hence $N$ is not simple. By the same argument as in the paragraph where we considered $|G|=168$, we conclude that $N\cong (C_2\times C_2\times C_2)\rtimes (C_7 \;{\rm sd}\; C_3)$. By the same reasoning as in the paragraph where we considered $|G|=112$, we conclude that
$G=NZ\cong (C_2\times C_2\times C_2)\rtimes (C_7 \;{\rm sd}\; (C_3\rtimes C_2))$, and again get a contradiction by proving the action of $z$ must be trivial either on a Sylow 7-subgroup or on the base group $C_2\times C_2\times C_2$.
 
Suppose finally that $|G|=2520$ or $5040$. Then $|H|=315$. If no Sylow subgroup of $H$ is normal in $H$ then a simple count shows that $H$ has 7 Sylow $3$-subgroups, 21 Sylow $5$-groups and $15$ Sylow $7$-subgroups, from which it follows quickly that $H$ has an element $\alpha$ of composite order involving at least two different primes. If any Sylow subgroup of $H$ is normal in $H$, again it follows quickly that $H$ has an element $\alpha$ of composite order involving at least two different primes. If $|\alpha|$ is not divisible by $15$ then $H$ has a subgroup isomorphic to $C_3\times C_7$ or $C_5\times C_7$ of minimal degree larger than $8$, a contradiction.  Hence $|\alpha|$ is divisible by $15$ so that $H$ has a subgroup isomorphic to $C_3\times C_5$ of minimal degree $8$, so that $G\in \C$. 

This completes the proof of the theorem.
\end{proof}

Combining results so far we can prove the following stepping-stone towards our main theorem below (Theorem \ref{theorem:degree9}).

\begin{theorem} \label{theorem:no8}
If $G$ and $H$ are groups such that $\mu(G\times H)\le 8$
then $\mu(G \times H) = \mu(G) + \mu(H)$.
\end{theorem}
\begin{proof}
Suppose by way of contradiction, that there exist subgroups $G$ and $H$ of $\Sym(8)$ such that $\langle G, H \rangle=GH$ is an internal direct
product and $\mu(G \times H) < \mu(G) + \mu(H)$. Certainly $G$ and $H$ are nontrivial. By Lemma \ref{lemma:wright}, it is not the case that
both $G$ and $H$ lie in $\mathscr{C}$. Without loss of generality, we may suppose that $G \not\in \mathscr{C}$. By Proposition \ref{proposition:class456},
$\mu(G) \geq 7$. If $\mu(G)=8$ then, 
by Proposition \ref{proposition:class8} and
Theorem \ref{theorem:class8}, $G$ is isomorphic to $K$ or $L$, described in the preamble before Proposition \ref{proposition:class8}, so $H$ is
trivial (since $C_{\Sym(8)}(K)$ and $C_{\Sym(8)}(L)$ are both trivial),
a contradiction. Hence $\mu(G)=7$. 

By Theorem \ref{theorem:G7}, $G\cong \langle a,b\,|\,a^3=b^4=1, a^b=a^{-1} \rangle$. Without loss of generality, we may take $a=(1\,2\,3)$
or $(1\,2\,3)(4\,5\,6)$ and $b$ a permutation of order $4$ that inverts $a$ by conjugation. By a straightforward calculation, the only
permutations of $\Sym(8)$ that invert $(1\,2\,3)(4\,5\,6)$ by conjugation have order $2$ or $6$, contradicting that $|b|=4$.
Hence
$a=(1\,2\,3)$. Clearly $b$ must be one of $(1\,2)\sigma$, $(1\,3)\sigma$ or $(2\,3)\sigma \}$, where $\sigma$ is a
$4$-cycle that fixes $1$, $2$ and $3$. Without loss of generality, $b=(1\,3)(4\,5\,6\,7)$  and $G=\langle a, b\rangle$. Clearly $C_{\Sym(8)}(G)=\langle (4\,5\,6\,7)\rangle$ and $C_{\Sym(8)}(G) \cap G= \langle (4\,6)(5\,7)\rangle$. But $\langle G,H \rangle=GH$ is an internal direct product, so $H
\cap G= \{1\}$ and $H \leq C_{\Sym(8)}(G)$. It follows quickly that $H$ is trivial, again a contradiction.
\end{proof}

\subsection{The $\Sym(9)$ Case}

Again we consider in turn transitive and intransitive embeddings, though in both cases now there are groups that fall outside Wright's class $\C$. We show directly that every nontrivial subgroup of the centraliser intersects nontrivially with our minimally embedded group.

\begin{proposition} \label{proposition:intrans9}
Let $G$ be a group such that $\mu(G)=9$
and its minimally embedded image in $\Sym(9)$ is intransitive.
Identify $G$ with its embedded image and let $C:=C_{\Sym(9)}(G)$.
Then every nontrivial subgroup of $C$ intersects $G$
nontrivially.
\end{proposition}

\begin{proof}
If at any stage we conclude $G\in\C$ then we are done by Corollary \ref{corollary:centraliser2}
and Proposition \ref{proposition:centraliser}.
Without loss of generality, we only need to consider the following three cases.
\newline

\noindent \underline {\textbf{Case (a):}}\, $G$ has an orbit $\{8,9\}$. \newline
\noindent By Lemma \ref{lemma:L1}, $G \cong C_2 \times H$ where $H \leq \Sym(\{1,2,3,4,5,6,7\})$ and $\mu(H)=7$. 
If $H\in \C$, then $G\in\C$, and we are done. Otherwise, by Theorem \ref{theorem:G7}, without loss of generality, we may suppose that
$$
G = \langle (1\,2\,3), (1\,3)(4\,5\,6\,7), (8\,9) \rangle\;.
$$
Then
$C=\langle (4\,5\,6\,7), (8\,9) \rangle$ and $G \cap C= \langle (4\,6)(5\,7), (8\,9) \rangle$. It quickly follows that every non-trivial
subgroup of $C$ intersects non-trivially with $G$. \newline

\noindent \underline {\textbf{Case (b):}}\, $G$ has an orbit $\{7, 8, 9\}$ and no orbit of size two. \newline

\noindent \underline {Subcase (i):}\, $G$ has two other orbits both of size $3$. \newline
\noindent By Lemma \ref{lemma:L1} (iv), $G$ contains a copy of $ C_3 \times C_3 \times C_3$, so $G \in \mathscr{C}$,
and we are done. \newline

\noindent \underline {Subcase (ii):}\, $G$ has an orbit $\{1,\ldots, 6\}$. \newline
\noindent We may regard $G$ as a subgroup of $\Sym(\{1,\ldots,6\}) \times \Sym(\{7,8,9\})$.  Let $\pi_1$ and $\pi_2$ be projections onto
$\Sym(\{1,\ldots,6\})$ and $\Sym(\{7,8,9\})$ respectively. 
Let $K_1 = \ker \pi_1|_G$ and $K_2 =\ker \pi_2|_G$ and observe that 
$\langle K_1, K_2 \rangle = K_1 K_2$ is an internal direct product.
By Lemma \ref{lemma:L1} (iv), we have $(7\, 8\, 9) \in G$. If moreover we have $G\pi_2=\langle (7\,8\,9) \rangle$ 
or that $G$ contains a $2$-cycle supported only on $\{7,8,9\}$, then $G \cong H \times K$ where $\mu(H)=6$ and $\mu(K)=3$, so that $H,K\in \C$, by
Proposition \ref{proposition:class456}, whence $G \in \mathscr{C}$, and we are done. 
Hence we may assume that $G\pi_2 = \Sym(\{7,8,9\})$ and that $G$ does not contain any $2$-cycle supported only on $\{7,8,9\}$. 
Therefore, $C$ is a subgroup of $\Sym(\{1,\ldots,6\})$.
Let $H_1$ be the stabiliser of the letter $1$, so that $H_1$ has index $6$ in $G$. By Theorem \ref{theorem:centraliser} and Proposition \ref{proposition:fix},
$|C|=|N_G(H_1)/H_1|=|\Fix(H_1)|= 1$, $2$, $3$ or $6$. If $C$ is trivial then we are done. 
If $|C|= 6$ then $H_1\le K_2$ and it follows that $|G|=18$ or $36$, so that, from the Appendix, $G\cong D_{18}$, whence
$G\in\C$, and we are done. We may suppose therefore that $|C|=2$ or $3$.
By Proposition \ref{proposition:fix} and Proposition \ref{proposition:block}, 
$G\pi_{1}$ embeds in $C_3 \wr \Sym(2)$ or $C_2 \wr \Sym(3)$ as a transitive subgroup.

Suppose first that $G\pi_{1}$ embeds inside $C_3 \wr \Sym(2)$. 
By Lemma \ref{lemma:L1} (i), $\mu(G\pi_{1})=6$. By the classification of transitive subgroups of $\Sym(6)$ (see \cite[Table 2.1]{DM96}),
it follows that $9$ divides $|G\pi_{1}|$. Therefore $27$ divides $|G|$ and $G$ contains a Sylow $3$-subgroup with 
minimal degree $9$. Hence $G \in \mathscr{C}$, and we are done.

Now suppose that $G\pi_{1}$ embeds inside 
\begin{equation}C_2 \wr \Sym(3)\cong  C_2 \times \Sym(4)\label{equation:w3}. \end{equation}
Then $C$ has order $2$ and may be identified under
this isomorphism with the
factor $C_2$ in the second group.
Let $z$ be the generator of $C$; we will show that $z\in G$.
Certainly there is some $\sigma  \in \Sym(\{7, 8, 9\})$ such that $\gamma:=z\sigma \in G$. 
We will show that $\sigma$ has order 1 or 3.
Suppose to the contrary that $\sigma$ has order $2$. 
Since $G\pi_2 \cong \Sym(3)$, we have $|G|=6|K_2|$. On the other hand, since no $2$-cycle supported only on $\{7,8,9\}$ is contained in $G$, 
we have $K_1= \langle (7\,8\,9) \rangle \cong C_3$ and so $|G|=3|G \pi_1|$. Therefore, $|G\pi_1|=2|K_2|$,
and so $|G|=2|K_1||K_2|$.  
Observe that $\gamma\not\in K_1K_2$ and $\gamma$ centralises $K_2$ and normalises $K_1$.
Upon comparing orders,
$$
G=\langle K_1, K_2, \gamma \rangle=K_1K_2\langle\gamma\rangle\cong K_2 \times (K_1 \rtimes \langle \gamma \rangle)\;.
$$
Since $K_1 = \langle (7\,8\,9) \rangle$ and $(7\,8\,9)^{\gamma}=(7\,9\,8)$, we have $K_1 \rtimes
\langle \gamma \rangle \cong \Sym(3)$, and since $K_2$ is isomorphic to a subgroup of $\Sym(4)$, by \eqref{equation:w3},
$G$ is isomorphic to a subgroup of $\Sym(3)
\times \Sym(4)$. Therefore $\mu(G) \leq \mu(\Sym(4) \times \Sym(3))=7$, contradicting that $\mu(G)=9$. 
Hence $\sigma$ has order 1 or 3, and it follows immediately that $z\in G$. Hence $C\le G$.\newline

\noindent \underline {\textbf{Case (c):}}\, $G$ has orbits $\{1, 2, 3, 4\}$ and $\{5, 6, 7 , 8, 9\}$. \newline
\noindent Let $\pi_1$ and $\pi_2$ be projections onto $\Sym(\{1, 2, 3, 4\})$ and $\Sym(\{5,\ldots, 9\})$ respectively. 
As before, let $K_1 = \ker \pi_1|_G$ and $K_2 =\ker \pi_2|_G$.
By Lemma \ref{lemma:L1} (v), without loss of generality, there is some $\gamma:=(5\,6\,7\,8\,9) \in G$. If $G\pi_{2}=\langle \gamma \rangle$, then $G \cong G\pi_1 \times C_5$ and $\mu(G\pi_1)=4$ by Lemma \ref{lemma:L1} (i),
so that $\mu(G\pi_1)\in\C$ 
by Proposition \ref{proposition:class456},
whence $G \in \mathscr{C}$, and we are done.  Therefore, $G\pi_{2}$ strictly contains $\langle \gamma \rangle$ and it follows that $C \leq \Sym(\{1, 2,
3, 4\})$.
If $G\pi_1=\Alt(\{1,2,3,4\})$ or $\Sym(\{1,2,3,4\})$, then $C=\{1\}$, and we are done. Therefore we may assume that
 $G\pi_1$ is isomorphic to $C_4$, $C_2 \times C_2$ or $D_8$.
 
Suppose first that $G\pi_{1}\cong C_4$. Without loss of generality, $G\pi_1$ is generated by $(1\,2\,3\,4)$, so that $C=G\pi_1$. By Lemma \ref{lemma:L1} (ii),
$(1\,2\,3\,4)\in G$ or $(1\,3)(2\,4)\in G$, and so $G \cap C \neq \{1\}$.

Now suppose that $G\pi_1\cong C_2\times C_2$, so that
$G\pi_{1}=C=\langle (1\,2)(3\,4), (1\,3)(2\,4) \rangle$. 
By Lemma \ref{lemma:L1}, we may suppose, without loss of generality, that $\alpha:=(1\,2)(3\,4) \in G$. 
We claim that $\beta:=(1\,3)(2\,4) \in G$.
Certainly $\beta\sigma \in G$ for some $\sigma \in \Sym(\{5,\ldots, 9\})$ and if $\sigma\in G$ or $|\sigma|$ is coprime with $|\beta|$, then $\beta \in G$ and
we are done. 
If $|\sigma|=6$ then we may replace $\sigma$ by $\sigma^3$.
Thus we may suppose that $|\sigma|=2$ or $4$ and that $\sigma\not\in G$.
Observe that $K_2=\langle \alpha \rangle$. On the one hand,
$|G|=|K_2||G\pi_2|=2|G\pi_2|$, and on the other, $|G|=|K_1||G\pi_1|=4|K_1|$. 
Therefore $|G|=2|K_1||K_2|$. Observe that $\langle K_1, K_2
\rangle=K_1 K_2$ is an internal direct product, 
$\beta\sigma\not\in K_1K_2$ and
$\beta\sigma$ centralises $K_2$ and normalises $K_1$.
Therefore,  
comparing orders, we have 
$$
G=\langle
\beta\sigma, K_1, K_2 \rangle=K_1K_2\langle\beta\sigma\rangle\cong
K_2 \times (K_1 \langle \beta\sigma \rangle)\cong
K_2 \times (K_1 \langle \sigma\rangle).
$$ 
But $K_2\cong C_2$ and $K_1\langle \sigma\rangle$ is a subgroup of
$\Sym(\{5,\ldots, 9\})$. Hence $G$ embeds in $C_2\times\Sym(5)$, so $\mu(G)\le 7$, contradicting that $\mu(G)=9$, and we are done.

Finally suppose that $G\pi_{1}\cong D_8$.
Without loss of generality, $G\pi_1=\langle r,s\rangle$
and $C=\langle r^2\rangle$ where
$r:=(1\,2\,3\,4)$ and $s:=(1\,2)(3\,4)$. 
We claim that $r^2 \in G$. 
This is immediate if $r \in G$, so we suppose that $r \not\in G$. 
Certainly, $r\sigma \in G$ for some $\sigma \in \Sym(\{5,\ldots, 9\})$
such that $\sigma\not\in G$ and $|\sigma|$ is divisible by $2$. 
If $|\sigma|=2$ or $6$ then $r^2=(r\sigma)^{|\sigma|}\in G$ and we are done.
Thus, we may suppose $|\sigma|=4$.
By Lemma \ref{lemma:L1} (ii), since
$r,r^{-1}\not\in G$, we have $r^{j}s \in G$ for some $j$. Hence
$$
\sigma^2=(r^{j-1}s\sigma)^2=(r^jsr\sigma)^2\in G,
$$ 
so that $r^2=(r\sigma)^{2}\sigma^{2} \in G$, and we are done.
\end{proof}

\begin{remark}
It can be verified by Magma that the only groups minimally embedded intransitively in $\Sym(9)$ that are not contained in $\mathscr{C}$ have orbits of size $2$ and $7$, or orbits of size $4$ and $5$. We do not prove this here as we do not need it for our main theorem. For more details the reader is referred to \cite{S10}.
\end{remark}

We have shown that if $G$ is a minimally embedded intransitive subgroup of $\Sym(9)$, then there is no subgroup $H$ of $\Sym(9)$ that
centralises $G$ such that $\mu(G \times H) < \mu(G) + \mu(H)$. Before we deal with the transitive case we observe the following lemma, whose proof is a straightforward direct calculation.

\begin{lemma} \label{lemma:diagonal9}
Let $W=C_3 \wr \Sym(3)$ and let the base group $B$ be generated by $x_1, x_2, x_3$. Let $U=\{ x_{1}^{i}x_{2}^{j}x_{3}^{k} \in B\,|\,i+j+k
\equiv 0\,\text{mod}\,3 \}$ and $V=\langle x_1x_2x_3 \rangle$. Then $V \subset U$ and $U$ and $V$ are the only non-trivial normal
subgroups of $W$ strictly contained in $B$.
\end{lemma}

\begin{proposition} \label{proposition:trans9}
Let $G$ be a group such that $\mu(G)=9$ and its minimally embedded image in $\Sym(9)$ is transitive. Identify $G$ with its image
and put $C:=C_{\Sym(9)}(G) \neq \{1\}$. Then $C \leq G$.
\end{proposition}

\begin{proof}
We may assume $C$ is nontrivial and also that that $G$ is non-abelian, for otherwise, $G=C$ by \cite[Theorem 4.2A]{DM96}. Let $H$ be a core-free subgroup of $G$ that affords the
minimal faithful representation. By Theorem \ref{theorem:centraliser}, $C \cong N_{G}(H)/H$ and since $|G:H|=9$, $|N_{G}(H):H|=3$ and so $C
\cong C_3$. By Proposition \ref{proposition:fix}, $\Fix(H)$ is a block on which the induced permutation group acts regularly, so by Proposition
\ref{proposition:block}, $G$ embeds inside the wreath product $C_3 \wr \Sym(3)$. Let $\pi$ be the projection of $G$ onto the top group $\Sym(3)$. Now $\ker \pi$  is contained in the base group and so must be a $3$-group. Since $G$ is transitive on blocks, $G\pi$ has order $3$ or $6$.  If $|G\pi|=3$, then $G$ is a $3$-group and so, by Corollary \ref{corollary:centraliser}, $C \leq G$. If $|G\pi|=6$ then $\pi$ is surjective and since $\ker \pi$ is a normal subgroup of $G$ contained in the base group $B$ it is normalised by $\Sym(3)$. By
Lemma \ref{lemma:diagonal9}, this kernel must contain $V$, which is cyclic of order $3$ and central in $G$. Therefore $V=C$, and once
again $C\le G$.
\end{proof}

\begin{remark}
It can be verified that there are, up to isomorphism, $3$ transitive groups minimally embedded in $\Sym(9)$ not contained in $\mathscr{C}$. Again, we do not prove this here as we do not need it for our main theorem, and for more details the reader is referred to
\cite{S10}.
\end{remark}

Combining the results above we can now prove our main theorem:

\begin{theorem}\label{theorem:degree9}
If $G$ and $H$ are groups such that $\mu(G\times H)\le 9$
then $\mu(G \times H) = \mu(G) + \mu(H)$.
\end{theorem}

\begin{proof}
Suppose by way of contradiction that there exist nontrivial subgroups $G$, $H$ of $\Sym(9)$ such that $\langle G, H \rangle=GH$ is an internal direct product
and $\mu(G \times H) < \mu(G) +\mu(H)$. 
Without loss of generality, we may suppose $G \not\in \mathscr{C}$.
By Proposition \ref{proposition:class456}, $\mu(G)
\geq 7$. If $\mu(G)=9$, then, by Proposition \ref{proposition:intrans9} and Proposition \ref{proposition:trans9}, $H$ intersects $G$
nontrivially, contradicting that $GH$ is an internal direct product. Hence $\mu(G)=7$ or $\mu(G)=8$. 

Suppose that $\mu(G)=7$. By Theorem \ref{theorem:G7} 
and \eqref{H}, $G \cong \langle a,b\,|\,a^3=b^4=1, a^b=a^{-1} \rangle$. Without loss of
generality, we may take $a=(1\,2\,3), (1\,2\,3)(4\,5\,6)$ or $(1\,2\,3)(4\,5\,6)(7\,8\,9)$, and $b$ to be a permutation of order $4$
that inverts $a$ by conjugation. By straightforward calculations, the only permutations in $\Sym(9)$ that invert $(1\,2\,3)(4\,5\,6)$ or $(1
\ 2\,3)(4\,5\,6)(7\,8\,9)$ have order $2$ or $6$, contradicting that $b$ has order $4$. Hence $a=(1\,2\,3)$. 
 Without loss of generality, $b=(1\,3)(4\,5\,6\,7)$ or $(1\,3)(4\,5\,6
\ 7)(8\,9)$ and $G=\langle a, b \rangle$. Clearly, in either case,  $$C_{\Sym(9)}(G)=\langle (4\,5\,6\,7), (8\,9) \rangle.$$ But $GH$ is an internal direct product, so $H \cap G=\{1\}$ and $H \leq C_{\Sym(9)}(G)$. Since $H$ is non-trivial, it follows
that $H=\langle (8\,9) \rangle$ or $H=\langle (4\,6)(5\,7)(8\,9) \rangle$. In both cases, $H \cong C_2$ so $\mu(H)=2$, giving $$\mu(G \times
H)< \mu(G)+\mu(H)=7+2=9.$$ Hence $\mu(G \times H) \leq 8$,
contradicting Theorem \ref{theorem:no8}.

Thus $\mu(G)=8$. By Theorem \ref{theorem:class8},  $G$ contains a copy of the group 
$K\cong (C_2\times C_2\times C_2)\rtimes C_7$ described explicitly in the preamble preceding Proposition \ref{proposition:class8}. 
All elements of the base group different from 1 are conjugate.
We may take the generators of the base group to be
$x,y$ and $z$ and the generator corresponding to the copy of $C_7$ to be $t$, and then conjugation by $t$ yields the following mapping: $$x \mapsto y \mapsto z \mapsto xy \mapsto yz
\mapsto xyz \mapsto xz \mapsto x.$$ 
Suppose first that $x$ is not a product of 4 disjoint 2-cycles. Without loss of generality we have the following three cases.

\noindent \underline{\textbf{Case (i):}}\,$x=(1\,2)$. \newline
Then $y=(a\,b)$ commutes with $x$ and so is disjoint from $x$, and so $xy=(1\,2)(a\,b)$ is not conjugate to $x$. \newline

\noindent \underline{\textbf{Case (ii):}}\,$x=(1\,2)(3\,4)$. \newline
Without loss of generality, $y=(1\,3)(2\,4)$, $(1\,2)(5\,6)$ or $(5\,6)(7\,8)$. 
If $y=(5\,6)(7\,8)$ then $xy=(1\,2)(3\,4)(5\,6)(7\,8)$ is not conjugate to $x$.
If $y=(1\,3)(2\,4)$ then, 
without loss of generality, $z=(5\,6)(7\,8)$, so $xz$ is not conjugate to $x$. 
If $y=(1\,2)(5\,6)$ then
$z=(1\,2)(7\,8), (3\,4)(7\,8)$ or $(5\,6)(7\,8)$ so that $(1\,2)(3
\ 4)(5\,6)(7\,8)=xyz$, $yz$ or $xz$ respectively is not conjugate to $x$.\newline

\noindent \underline{\textbf{Case (iii):}}\,$x=(1\,2)(3\,4)(5\,6)$. \newline
Without loss of generality, $y=(1\,2)(3\,4)(7\,8)$, $(1\,3)(2\,4)(5\,6)$ or $(1\,3)(2\,4)(7\,8)$. Then $xy=(5\,6)(7\,8)$, $(1\,4)(2 \
3)$ or $(1\,4)(2\,3)(5\,6)(7\,8)$ respectively is not conjugate to $x$. \vspace{0pt}

\noindent
All of these cases lead to a contradiction, so, without loss of generality,
$x=(1\,2)(3\,4)(5\,6)(7\,8)$. Now both $y$ and $z$ fix $9$ since $x=x^y=x^z$. If $t$ moves
$9$, then $x^t$ moves $9$, contradicting that $x^t=z$ fixes $9$. Hence $t$ also fixes $9$ and so $t$ is a $7$-cycle permuting letters amongst
$\{1,\ldots,8\}$. Therefore, $t$ fixes another letter and so, without loss of generality, $t$ is a $7$-cycle permuting $1,\ldots, 7$ in some
order. Let $w \in C:=C_{\Sym(9)}(G)$. Then $w$ commutes with $t$, so $w=t^r$, or $w=t^{r}(8\,9)$ for some $r$. But if
$w=t^{r}(8\,9)$, then $x^w$ moves $9$, so $x^w \neq x$, contradicting that $w$ commutes with $x$. Hence $w=t^r$, which implies $w=1$, since non-trivial
powers of $t$ do not commute with $x$. Thus $H \leq C=\{1\}$, so $H$ is trivial, a contradiction. This completes the proof of the theorem.
\end{proof}

The string of results above show that there are no examples of groups $G$ and $H$ such that $\mu(G \times H) \leq 9$ and \begin{equation} \mu(G \times H) < \mu(G) + \mu(H). \label{equation:strict} \end{equation}   
In $\Sym(10)$, however, $G$ can be taken to be any split extension of the deleted permutation module for $\Sym(5)$ over $\mathbb{F}_2$ by a subgroup that contains an element of order $5$. It is well-known that there are 5 such choices for the top group of the split extension, namely  $C_5, D_5, C_5 \rtimes C_4, \Alt(5)$ or $\Sym(5)$. So, for the example of smallest order, one takes $G$ to be $(C_2 \times C_2 \times C_2 \times C_2) \rtimes C_5$ and $H$ to be $C_2$ (its centraliser in $\Sym(10)$), and all examples have the property that  $\mu(G)=\mu(G \times H) < \mu(G) + \mu(H)$. \vspace{0pt}

The authors are not aware of any
examples of groups $G$ and $H$ that do not decompose as nontrivial direct products for which \begin{equation} \max\{\mu(G),\mu(H)\} < \mu(G \times H) < \mu(G) + \mu(H). \label{equation:cascade} \end{equation}  
One can easily transform \eqref{equation:strict} into an infinite class of examples of \eqref{equation:cascade} by taking direct products with a new group of order coprime to both $G$ and $H$ (see \cite[Section 7]{S14}). 

\section{acknowledgements}
The authors thank Cheryl Praeger for helpful discussions in catalysing the proofs in some of the results above, particularly in the transitive cases. \vspace{0pt}

Part of this work was completed while the second author was on a research visit to the University of Sydney. He sincerely thanks Anthony Henderson for supporting this research visit. 

\newpage

\appendix
\section{\\ Minimal Degrees of Groups of Small Order} \label{App:AppA}
Below is a table of minimal permutation degrees of groups of small order which we rely on in the article. Many of these calculations can be done by hand, or can easily be verified using Magma \cite{BCFS10}. The group identification is done using the SmallGroupsLibrary in Magma.

\begin{center}
\begin{longtable}{|l|l|c|l|}
\caption{Minimal Degrees of Groups of Small Order}\\
\hline
\textbf{Group ID} & \textbf{Stucture} & \textbf{Minimal Degree} & \textbf{Contained in $\mathscr{C}$?} \\
\hline
\endfirsthead
\multicolumn{4}{c}%
{\tablename\ \thetable\ -- \textit{Continued from previous page}} \\
\hline
\textbf{Group ID} & \textbf{Stucture} & \textbf{Minimal Degree} & \textbf{Contained in $\mathscr{C}$?} \\

\hline
\endhead
\hline \multicolumn{4}{r}{\textit{Continued on next page}} \\
\endfoot
\hline
\endlastfoot
$(56,1)				$     &           $	[C_7 : C_8]$   								&           $	15	$              &  No 									\\	
$(56,2)				$     &           $Ê	[C_{56}]$    								&           $Ê	15	$              &  Yes Abelian								\\
$(56,3)				$     &           $	[C_7 : Q_8]$								&           $Ê	15	$              &  No										\\
$(56,4)				$     &           $	[C_4 \times D_{14}]$							&           $	11	$              &  Yes									\\ 
$(56,5)				$     &           $	[D_{56}]$			    						&           $	11	$              &  Yes									\\
$(56,6)				$     &           $	[C_2 \times (C_7 : C_4)]	$	    				&           $	13	$              &  No										\\
$(56,7)Ê				$     &           $	[(C_{14} \times C_2) : C_2]$	 	   			&           $	11	$              &  Yes									\\
$(56,8)				$     &           $	[C_{28} \times C_2]$		    					&           $	13	$              &  Yes	Abelian							\\
$(56,9)				$     &           $	[C_7 \times D_8]$			    				&           $	11	$              &  Yes	Nilpotent							\\
$(56,10)				$     &           $	[C_7 \times Q_8]  $		   	 				&	     $	15	$              &  Yes	Nilpotent							\\
$(56,11)				$     &           $	[(C_2 \times C_2 \times C_2) : C_7]$   	 		&           $Ê	8	$              &  No										\\
$(56,12)Ê				$     &           $	[C_2 \times C_2 \times D_{14}]$   		 		&           $Ê	11	$              &  Yes									\\	
$(56,13)				$     &           $	[C_{14} \times C_2 \times C_2]$      				&           $Ê13	$              &  Yes	Abelian							\\	
$(54,1)				$     &           $	[D_{54}]$    				     		 		&           $Ê	27	$              &  Yes									\\	
$(54,2)				$     &           $	[C_{54}]$									&           $Ê	29	$              &  Yes	Abelian							\\
$(54,3)				$     &           $	[C_3 \times D_{18}]$     						&           $Ê	12	$              &  Yes									\\
$(54,4)				$     &           $	[C_9 \times \Sym(3)]$  						&           $Ê	12	$              &  Yes									\\
$(54,5)				$     &           $	[((C_3 \times C_3) : C_3) : C_2]$				&           $Ê	9	$              &  Yes									\\
$(54,6)				$     &           $	[(C_9 : C_3) : C_2]$ 	  						&           $Ê	9	$              &  Yes									\\
$(54,7)				$     &           $	[(C_9 \times C_3) : C_2]$     					&           $Ê	12	$              &  Yes									\\
$(54,8)				$     &           $	[((C_3 \times C_3) : C_3) : C_2]$    				&           $Ê	9	$              &  Yes									\\
$(54,9)				$     &           $	[C_{18} \times C_3]$     						&  	     $Ê	14	$              &  Yes	Abelian							\\
$(54,10)				$     &           $	[C_2 \times ((C_3 \times C_3) : C_3)]$     			&          $Ê	11	$              &  Yes	Nilpotent							\\
$(54,11)				$     &           $	[C_2 \times (C_9 : C_3)]$   					&           $Ê	11	$              &  Yes	Nilpotent							\\
$(54,12)Ê				$     &           $	[C_3 \times C_3 \times \Sym(3)]$     				&           $Ê	9	$              &  Yes									\\
$(54,13)				$     &           $	[C_3 \times ((C_3 \times C_3) : C_2)]$ 			&           $Ê	9	$              &  Yes									\\
$(54,14) 				$     &           $ 	[(C_3 \times C_3 \times C_3) : C_2]$			&           $Ê	9 	$              &  Yes									\\
$(54,15)				$     &           $	[C_6 \times C_3 \times C_3]$     				&           $Ê	11	$              &  Yes	Abelian							\\
$(48,1)				$     &           $Ê	[C_3 : C_{16}]$    							&           $Ê	19	$              &  No										\\
$(48,2) 				$     &           $	[C_{48}]$    								&           $Ê	19	$              &  Yes	Abelian							\\
$(48,3)Ê				$     &           $	[(C_4 \times C_4) : C_3]$     					&           $Ê	12	$              &  No										\\
$(48,4)Ê				$     &           $	[C_8 \times \Sym(3)]$     						&           $Ê	11	$              &  Yes									\\
$(48,5)Ê				$     &           $	[C_{24} : C_2]$			     					&           $Ê	11	$              &  Yes									\\
$(48,6)Ê				$     &           $	[C_{24} : C_2]$     							&           $Ê	11	$             &  Yes										\\
$(48,7)Ê				$     &           $	[D_{48}]$     								&           $Ê	11	$              &  Yes									\\
$(48,8)Ê				$     &           $	[C_3 : Q_{16}]$     							&           $Ê	19	$              &  No										\\
$(48,9)				$     &           $Ê	[C_2 \times (C_3 : C_8)]$	     					&           $Ê	13	$              &  No										\\
$(48,10)				$     &           $Ê	[(C_3 : C_8) : C_2]$    						&           $Ê	11	$              &  No										\\	
$(48,11)	Ê			$     &           $	[C_4 \times (C_3 : C_4)]$         		     			&           $Ê	11	$              &  No										\\
$(48,12)	Ê			$     &           $	[(C_3 : C_4) : C_4]$     						&           $Ê	11	$              &  No										\\	
$(48,13)	Ê			$     &           $	[C_{12} : C_4]$     							&           $Ê	11	$              &  No										\\
$(48,14)				$     &           $	[(C_{12} \times C_2) : C_2]$    					&	     $Ê	11	$              &  No										\\
$(48,15)	Ê			$     &           $	[(C_3 \times D_8) : C_2]$ 						&           $Ê	11	$              &  No										\\
$(48,16)	Ê			$     &           $	[(C_3 : C_8) : C_2]$ 							&           $Ê	11	$              &  No										\\
$(48,17)	Ê			$     &           $	[(C_3 \times Q_8) : C_2] $ 			       		&           $Ê	11	$              &  No										\\
$(48,18)				$     &           $Ê	[C_3 : Q_{16}]$     							&           $Ê	19	$              &  No										\\
$(48,19)	Ê			$     &           $	[(C_2 \times (C_3 : C_4)) : C_2]$	    	 		&           $Ê	11	$              &  No										\\
$(48,20)	Ê			$     &           $	[C_{12} \times C_4]$		    	 				&           $Ê	11	$              &  Yes	Abelian							\\
$(48,21)	Ê			$     &           $	[C_3 \times ((C_4 \times C_2) : C_2)]$   			&   		$Ê11	$              &  Yes	Nilpotent							\\	
$(48,22)	Ê			$     &           $	[C_3 \times (C_4 : C_4)]$     					&     		$Ê11	$              &  Yes	Nilpotent							\\
$(48,23)				$     &           $Ê	[C_{24} \times C_2]$     						&           $Ê	13	$              &  Yes	Abelian							\\
$(48,24)				$     &           $	[C_3 \times (C_8 : C_2)]$     					&     		$Ê11	$              &  Yes	Nilpotent							\\
$(48,25)	Ê			$     &           $	[C_3 \times D_{16}]$     						&           $Ê	11	$              &  Yes									\\	
$(48,26)	Ê			$     &           $	[C_3 \times QD_{16}]$     						&           $Ê	11	$              &  Yes									\\			
$(48,27)	Ê			$     &           $	[C_3 \times Q_{16}]$     						&           $Ê	19	$              &  Yes	Nilpotent							\\
$(48,28)	Ê			$     &           $ 	[SL(2,3) \rightarrow G \rightarrow C_2]$     		&        $Ê16	$              &  Yes 									\\
$(48,29)	Ê			$     &           $	[GL(2,3)]$    								&           $Ê	8	$              &  Yes									\\
$(48,30)	Ê			$     &           $	[\Alt(4) : C_4]$ 			    					&           $Ê	8	$              &  Yes									\\
$(48,31)	Ê			$     &           $	[C_4 \times \Alt(4)]$     						&           $Ê	8	$              &  Yes									\\
$(48,32)	Ê			$     &           $	[C_2 \times SL(2,3)]$     						&    		$Ê10	$              &  Yes									\\
$(48,33)	Ê			$     &           $	[SL(2,3) : C_2]$     							&     		$16	$	     	&  No									\\
$(48,34)	Ê			$     &           $	[C_2 \times (C_3 : Q_8)]$     					&        $13		$     		&  No									\\
$(48,35)				$     &           $	[C_2 \times C_4 \times \Sym(3)]$     				&      	  $Ê	9	$              &  Yes									\\
$(48,36)	Ê			$     &           $	[C_2 \times D_{24}]$    						&           $Ê	9	$              &  Yes									\\
$(48,37)	Ê			$     &           $	[(C_{12} \times C_2) : C_2]$     					&       	$11	$  		&  No									\\
$(48,38)	Ê			$     &           $	[D_8 \times \Sym(3)]$    						&		$Ê7$              	&  Yes									\\
$(48,39)	Ê			$     &           $	[(C_2 \times (C_3 : C_4)) : C_2]$ 	   			&           $Ê	11	$              &  No										\\
$(48,40)	Ê			$     &           $	[Q_8 \times \Sym(3)]$     						&           $Ê	11	$              &  Yes									\\
$(48,41)	Ê			$     &           $	[(C_4 \times \Sym(3)) : C_2]$     				&           $Ê	11	$              &  Yes									\\
$(48,42)	Ê			$     &           $	[C_2 \times C_2 \times (C_3 : C_4)]$     			&           $Ê	11	$              &  No										\\
$(48,43)				$     &           $Ê	[C_2 \times ((C_6 \times C_2) : C_2)]$		   	&           $Ê	9	$              &  Yes									\\
$(48,44)				$     &           $Ê	[C_{12} \times C_2 \times C_2]$  				&           $Ê	11	$              &  Yes	Abelian							\\
$(48,45)				$     &           $Ê	[C_6 \times D_8]$     						&           $Ê	9	$              &  Yes	Nilpotent							\\
$(48,46)				$     &           $Ê	[C_6 \times Q_8]$     						&           $Ê	13	$              &  Yes	Nilpotent							\\
$(48,47)				$     &           $Ê	[C_3 \times ((C_4 \times C_2) : C_2)]$	    		&           $Ê	11	$              &  Yes	Nilpotent							\\
$(48,48)				$     &           $Ê	[C_2 \times \Sym(4)]$     						&           $Ê	6	$              &  Yes									\\
$(48,49)				$     &           $Ê	[C_2 \times C_2 \times \Alt(4)]$  	   			&           $Ê	8	$              &  Yes									\\	
$(48,50)				$     &           $Ê	[(C_2 \times C_2 \times C_2 \times C_2) : C_3]$ 	&     	     $Ê	8	$    		&  Yes									\\
$(48,51)				$     &           $Ê	[C_2 \times C_2 \times C_2 \times \Sym(3)]$     	&           $Ê	9	$              &  Yes									\\
$(48,52)				$     &           $Ê	[C_6 \times C_2 \times C_2 \times C_2]$   	  	&           $	11	$              &  Yes	Abelian							\\
$(40,1)	Ê			$     &           $	[C_5 : C_8]$     								&           $Ê	13	$              &  No										\\
$(40,2)				$     &	    $	[C_{40}]$     								&           $Ê	13	$              &  Yes	Abelian							\\
$(40,3)	Ê			$     	&           $	[C_5 : C_8]$    				 				&           $Ê	13	$              &  No										\\
$(40,4)	Ê			$     	&           $	[C_5 : Q_8]$    								&           $Ê	13	$              &  No										\\
$(40,5)	Ê			$     	&           $	[C_4 \times D_{10}]$     						&           $Ê	9	$              &  Yes									\\
$(40,6)	Ê			$     	&           $	[D_{40}]$     								&           $Ê	9	$              &  Yes									\\
$(40,7)	Ê			$     	&           $	[C_2 \times (C_5 : C_4)]$     					&           $Ê	11	$              &  No										\\
$(40,8)	Ê			$     	&           $	[(C_{10} \times C_2) : C_2]$     					&           $Ê	9	$              &  Yes									\\	
$(40,9)Ê				$     	&           $	[C_{20} \times C_2]$     						&	      $Ê11	$              &  Yes	Abelian							\\
$(40,10)				$     	&           $Ê	[C_5 \times D_8]$     						&           $Ê	9	$              &  Yes	Nilpotent							\\
$ (40,11)				$     	&           $	[C_5 \times Q_8]$     						&           $Ê	13	$              &  Yes	Nilpotent							\\
$(40,12)				$     	&           $	[C_2 \times (C_5 : C_4)]$     					&           $Ê	7	$              &  Yes									\\
$(40,13)Ê				$     	&           $	[C_2 \times C_2 \times D_{10}]$     				&           $Ê	9	$              &  Yes									\\
$(40,14)				$     	&           $	[C_{10} \times C_2 \times C_2]$     				&           $Ê	11	$              &  Yes	Abelian							\\
$(36,1)	Ê			$     	&           $	[C_9 : C_4]$     								&           $Ê	13	$              &  No										\\		
$(36,2)	Ê			$     	&           $	[C_{36}]$     								&           $Ê	13	$              &  Yes	Abelian							\\		
$(36,3)	Ê			$     	&           $	[(C_2 \times C_2) : C_9]$     					&           $Ê	13	$              &  No										\\
$(36,4)	Ê			$     	&           $	[D_{36}]$     								&           $Ê	11	$              &  Yes									\\
$(36,5)	Ê			$     	&           $	[C_{18} \times C_2]$     						&           $Ê	13	$              &  Yes	Abelian							\\
$(36,6)	Ê			$     	&           $	[C_3 \times (C_3 : C_4)]$     					&           $Ê	10	$              &  No										\\
$(36,7)	Ê			$     	&           $	[(C_3 \times C_3) : C_4]$     					&           $Ê	10	$              &  No										\\
$(36,8)	Ê			$     	&           $	[C_{12} \times C_3]$     						&           $Ê	10	$		& Yes Abelian								\\
$(36,9)Ê				$     	&           $	[(C_3 \times C_3) : C_4]$	     					&           $Ê	6	$              &  Yes									\\
$(36,10)Ê				$     	&           $	[\Sym(3) \times \Sym(3)]$     					&           $Ê	6	$             	&  Yes									\\
$(36,11)Ê				$     	&           $	[C_3 \times \Alt(4)]$     						&           $Ê	7	$          	&  Yes									\\
$(36,12)				$     	&           $Ê	[C_6 \times \Sym(3)]$    				 		&           $Ê	8	$     		&  Yes									\\
$(36,13)				$     	&           $Ê	[C_2 \times ((C_3 \times C_3) : C_2)]$	   		&           $Ê	8	$             	&  Yes									\\
$(36,14)				$     	&           $Ê	[C_6 \times C_6]$     						&           $	10	$     	        &  Yes	Abelian							\\
$(32,1)Ê				$     	&           $	[C_{32}]$     								&           $Ê	32	$		& Yes 	Abelian							\\
$(32,2)Ê				$     	&           $	[(C_4 \times C_2) : C_4]$     					&           $Ê	12	$     	 	&  Yes	Nilpotent							\\
$(32,3)Ê				$     	&           $	[C_8 \times C_4]$     						&           $Ê	12	$		& Yes	Abelian							\\
$(32,4)Ê				$     	&           $	[C_8 : C_4]$     								&	     $Ê	12	$     	  	&  Yes	Nilpotent							\\
$(32,5)Ê				$     	&           $	[(C_8 \times C_2) : C_2]$     					&           $Ê	12	$     	  	&  Yes	Nilpotent							\\
$(32,6)Ê				$     	&           $	[((C_4 \times C_2) : C_2) : C_2]$     				&           $Ê	8	$     	  	&  Yes	Nilpotent							\\
$(32,7)Ê				$     	&           $	[(C_8 : C_2) : C_2]$     						&           $Ê	8	$  	     	&  Yes	Nilpotent							\\
$(32,8)Ê				$     	&           $	[C_2 . ((C_4 \times C_2) : C_2) = (C_2 \times C_2) . (C_4 \times C_2)]$     &           $Ê16 $              &  Yes	Nilpotent				\\
$(32,9)Ê				$     	&           $	[(C_8 \times C_2) : C_2]$     					&           $Ê	12	$              &  Yes	Nilpotent							\\
$(32,10)				$     	&           $Ê	[Q_8 : C_4]$     							&           $Ê	12	$              &  Yes	Nilpotent							\\
$(32,11)Ê				$     	&           $	[(C_4 \times C_4) : C_2]$     					&           $Ê	8	$              &  Yes	Nilpotent							\\
$(32,12)				$     	&           $Ê	[C_4 : C_8]$     								&      	$Ê	12	$              &  Yes	Nilpotent							\\
$(32,13)				$     	&           $Ê	[C_8 : C_4]$     								&           $Ê	12	$              &  Yes	Nilpotent							\\
$(32,14)				$     	&           $Ê	[C_8 : C_4]$     								&           $Ê	12	$              &  Yes	Nilpotent							\\
$(32,15)				$     	&           $Ê	[C_4 . D_8 = C_4 . (C_4 \times C_2)]$     			&           $Ê	16	$              &  Yes	Nilpotent							\\
$(32,16)Ê				$     	&           $	[C_{16} \times C_2]$    						&           $Ê	18$			& Yes	Abelian							\\
$(32,17)Ê				$     	&           $	[C_{16} : C_2]$     							&           $Ê	16	$              &  Yes	Nilpotent							\\
$(32,18)Ê				$     	&           $	[D_{32}]$     								&           $Ê	16	$              &  Yes	Nilpotent							\\
$(32,19)Ê				$     	&           $	[QD_{32}]$     								&           $Ê	16	$              &  Yes	Nilpotent							\\
$(32,20)Ê				$     	&           $	[Q_{32}]$     								&           $Ê	32	$              &  Yes	Nilpotent							\\
$(32,21)Ê				$     	&           $	[C_4 \times C_4 \times C_2]$   		  			&           $Ê	10$			& Yes	Abelian							\\
$(32,22)Ê				$     	&           $	[C_2 \times ((C_4 \times C_2) : C_2)]$		   	&           $Ê	10$              	&  Yes	Nilpotent							\\
$(32,23)Ê				$     	&           $	[C_2 \times (C_4 : C_4)]$     					&           $Ê	10 $              	&  Yes	Nilpotent							\\
$(32,24)Ê				$     	&           $	[(C_4 \times C_4) : C_2]$	 				    	&           $Ê	12$              	&  Yes	Nilpotent							\\
$(32,25)Ê				$     	&           $	[C_4 \times D_8]$    		 					&           $Ê	8$              	&  Yes	Nilpotent							\\
$(32,26)Ê				$     	&           $	[C_4 \times Q_8]$     						&           $Ê	12$             	&  Yes	Nilpotent							\\
$(32,27)Ê				$     	&           $	[(C_2 \times C_2 \times C_2 \times C_2) : C_2]$     	&           $Ê	8$              	&  Yes	Nilpotent							\\
$(32,28)Ê				$     	&           $	[(C_4 \times C_2 \times C_2) : C_2]$     			&           $Ê	8$              	&  Yes	Nilpotent							\\
$(32,29)				$     	&           $Ê	[(C_2 \times Q_8) : C_2]$     					&           $Ê	12$              	&  Yes	Nilpotent							\\
$(32,30)				$     	&           $Ê	[(C_4 \times C_2 \times C_2) : C_2]$     			&           $Ê	12$              	&  Yes	Nilpotent							\\
$(32,31)				$     	&           $Ê	[(C_4 \times C_4) : C_2]$     					&           $Ê	12$              	&  Yes	Nilpotent							\\
$(32,32)				$     	&           $Ê	[(C_2 \times C_2) . (C_2 \times C_2 \times C_2)]$    &  	     $Ê	16$              	&  Yes	Nilpotent							\\
$(32,33)				$     	&           $Ê	[(C_4 \times C_4) : C_2]$    	 				&           $Ê	16$              	&  Yes	Nilpotent							\\
$(32,34)				$     	&           $Ê	[(C_4 \times C_4) : C_2]$   	  				&           $Ê	8$              	&  Yes	Nilpotent							\\
$(32,35)				$     	&           $Ê	[C_4 : Q_8]$    				 				&           $Ê	12$              	&  Yes	Nilpotent							\\
$(32,36)				$     	&           $Ê	[C_8 \times C_2 \times C_2]$   		  			&           $Ê	12$			& Yes 	Abelian							\\
$(32,37)				$     	&           $Ê	[C_2 \times (C_8 : C_2)]$     					&           $Ê	10	$              	&  Yes	Nilpotent							\\
$(32,38)				$     	&           $Ê	[(C_8 \times C_2) : C_2]$     					&           $Ê	16	$              	&  Yes	Nilpotent							\\
$(32,39)				$     	&           $Ê	[C_2 \times D_{16}]$     						&           $Ê	10	$              	&  Yes	Nilpotent							\\
$(32,40)				$     	&           $Ê	[C_2 \times QD_{16}]$     						&           $Ê	10	$              	&  Yes	Nilpotent							\\
$(32,41)				$     	&           $Ê	[C_2 \times Q_{16}]$     						&           $Ê	18	$              	&  Yes	Nilpotent							\\
$(32,42)				$     	&           $Ê	[(C_8 \times C_2) : C_2]$     					&           $Ê	16	$              	&  Yes	Nilpotent							\\
$(32,43)				$     	&           $Ê	[(C_2 \times D_8) : C_2]$     					&           $Ê	8	$             	&  Yes	Nilpotent							\\
$(32,44)				$     	&           $Ê	[(C_2 \times Q_8) : C_2]$     					&           $Ê	16	$              	&  Yes	Nilpotent							\\
$(32,45)				$     	&           $Ê	[C_4 \times C_2 \times C_2 \times C_2]$		     	&           $Ê	10$			& Yes 	Abelian							\\
$(32,46)				$     	&           $Ê	[C_2 \times C_2 \times D_8]$     				&           $Ê	8	$              	&  Yes 	Nilpotent							\\
$(32,47)				$     	&           $Ê	[C_2 \times C_2 \times Q_8]$     				&           $Ê	12	$              	&  Yes	Nilpotent							\\
$(32,48)				$     	&           $Ê	[C_2 \times ((C_4 \times C_2) : C_2)]$	   		&           $Ê	10	$             	&  Yes	Nilpotent							\\
$(32,49)				$     	&           $Ê	[(C_2 \times D_8) : C_2]$     					&           $Ê	8$              	&  Yes	Nilpotent							\\
$(32,50)				$     	&           $Ê	[(C_2 \times Q_8) : C_2]$     					&           $Ê	16$              	&  Yes	Nilpotent							\\
$(32,51)				$     	&           $Ê	[C_2 \times C_2 \times C_2 \times C_2 \times C_2]$	&           $	10$			&	 Yes 	Abelian							\\
$(28,1)				$     	&           $Ê	[C_7 : C_4]$    				 				&           $Ê	11$              	&  No									\\
$(28,2)				$     	&           $Ê	[C_{28}]$     								&           $Ê	11$           		&  Yes	Abelian							\\
$(28,3)				$     	&           $Ê	[D_{28}]$    	 							&           $Ê	9$            		&  Yes									\\
$(28,4)				$     	&           $Ê	[C_{14} \times C_2]$     						&           $Ê	11$           		&  Yes	Abelian							\\
$(27,1)				$     	&           $Ê	[C_{27}]$     								&           $Ê	27$			&Yes 	Abelian							\\
$(27,2)				$     	&           $	[C_9 \times C_3]$     						&           $Ê	12$			& Yes 	Abelian							\\
$(27,3)Ê				$     	&           $	[(C_3 \times C_3) : C_3]$     					&           $Ê	9$ 	             	&  Yes	Nilpotent							\\
$(27,4)Ê				$     	&           $	[C_9 : C_3]$     								&           $Ê	9$	              	&  Yes	Nilpotent							\\
$(27,5)Ê				$     	&           $	[C_3 \times C_3 \times C_3]$   		 			&           $Ê	9$			& Yes	Abelian 							\\
$(24,1)Ê				$     	&      	     $	[C_3 : C_8]$   				  				&           $Ê	11	$              &  No										\\
$(24,2)Ê				$     	&	     $	[C_{24}]$     								&           $Ê	11	$              &  Yes	Abelian							\\
$(24,3)Ê				$     	&           $	[SL(2,3)]$    				 				&           $Ê	8	$              &  Yes									\\
$(24,4)				$     	&           $Ê	[C_3 : Q_8]$     							&           $Ê	11	$              &  No										\\
$(24,5)				$     	&           $Ê	[C_4 \times \Sym(3)]$ 	    					&           $Ê	7	$              &  Yes									\\
$(24,6)				$     	&           $Ê	[D_{24}]$     								&           $Ê	7	$              &  Yes									\\
$(24,7)				$     	&           $Ê	[C_2 \times (C_3 : C_4)]$    					&           $Ê	9	$              &  No										\\
$(24,8)Ê				$     	&           $	[(C_6 \times C_2) : C_2]$     					&           $Ê	7	$              &  Yes									\\
$(24,9)Ê				$     	&           $	[C_{12} \times C_2]$     						&           $Ê	9	$              &  Yes	Abelian							\\
$(24,10)				$     	&           $Ê	[C_3 \times D_8]$     						&           $Ê	7	$              &  Yes	Nilpotent							\\
$(24,11)				$     	&           $Ê	[C_3 \times Q_8]$     						&           $Ê	11	$              &  Yes	Nilpotent							\\
$(24,12)				$     	&           $Ê	[\Sym(4)]$     								&           $Ê	4	$              &  Yes									\\
$(24,13)				$     	&           $Ê	[C_2 \times \Alt(4)]$     						&           $Ê	6	$              &  Yes									\\
$(24,14)				$     	&           $Ê	[C_2 \times C_2 \times \Sym(3)]$     				&           $Ê	7	$              &  Yes									\\
$(24,15)				$     	&           $Ê	[C_6 \times C_2 \times C_2]$    			 	&           $Ê	9	$              &  Yes	Abelian							\\
$(20,1)Ê				$     	&           $	[C_5 : C_4]$     								&           $Ê	9	$              &  No										\\
$(20,2)Ê				$     	&           $	[C_{20}]$     								&           $Ê	9	$              &  Yes	Abelian							\\
$(20,3)Ê				$     	&           $	[C_5 : C_4]$     								&           $Ê	5	$              &  Yes									\\
$(20,4)Ê				$     	&           $	[D_{20}]$    								 &           $Ê7$ 		   	&  Yes									\\
$(20,5)Ê				$     	&           $	[C_{10} \times C_2]$     						&           $Ê	9	$              	& Yes	Abelian							\\
$(18,1)				$     	&	$	[D_{18}]$ 									&	    $	9	$		&Yes										\\
$(18,2) 				$	&	$	[C_{18}] $									&		$11$			&Yes 	Abelian							\\
$(18,3)				$	&	$	[C_3 \times \Sym(3)]$ 						&		$6$			&Yes										\\
$(18,4) 				$	&	$	(C_3 \times C_3) : C_2]$ 						&		$6$			&Yes										\\
$(18,5) 				$	&	$	[C_6 \times C_3]$ 							&		$8$			&Yes 	Abelian							\\
$(16,1)Ê				$     	&           $	[C_{16}]$     								&           $Ê	16	$              &  Yes	Abelian							\\
$(16,2)Ê				$     	&           $	[C_4 \times C_4]$     						&           $Ê	8	$              &  Yes	Abelian							\\
$(16,3)Ê				$     	&           $	[(C_4 \times C_2) : C_2]$     					&           $Ê	8	$              &  Yes	Nilpotent							\\
$(16,4)Ê				$     	&           $	[C_4 : C_4]$     								&           $Ê	8	$              &  Yes	Nilpotent							\\
$(16,5)Ê				$     	&           $	[C_8 \times C_2]$     						&           $Ê	10	$              &  Yes	Abelian							\\
$(16,6)Ê				$     	&           $	[C_8 : C_2]$     								&           $Ê	8	$              &  Yes	Nilpotent							\\
$(16,7)Ê				$     	&           $	[D_{16}]$     								&           $Ê	8	$              &  Yes	Nilpotent							\\
$(16,8)Ê				$     	&           $	[QD_{16}]$     								&           $Ê	8	$              &  Yes	Nilpotent							\\
$(16,9)Ê				$     	&           $	[Q_{16}]$     								&           $Ê	16	$              &  Yes	Nilpotent							\\
$(16,10)				$     	&           $	[C_4 \times C_2 \times C_2]$	     				&           $Ê	8	$              &  Yes	Abelian							\\
$(16,11)Ê				$     	&           $	[C_2 \times D_8]$     						&           $Ê	6	$              &  Yes	Nilpotent							\\
$(16,12)Ê				$     	&           $	[C_2 \times Q_8]$     						&           $Ê	10	$              &  Yes	Nilpotent							\\
$(16,13)Ê				$    	&           $	[(C_4 \times C_2) : C_2]$     					&           $Ê	8	$              &  Yes	Nilpotent							\\
$(16,14)				$     	&           $Ê	[C_2 \times C_2 \times C_2 \times C_2]$     		&           $Ê	8	$              &  Yes	Abelian							\\
$(12,1)Ê				$     	&           $	[C_3 : C_4]$     								&           $Ê	7	$              &  No										\\
$(12,2)Ê				$     	&           $	[C_{12}]$    			 					&           $Ê	7	$              &  Yes	Abelian							\\
$(12,3)Ê				$     	&           $	[\Alt(4)]$     								&           $Ê	4	$              &  Yes									\\
$(12,4)Ê				$     	&           $	[D_{12}]$     								&           $Ê	5	$              &  Yes									\\
$(12,5)				$     	&           $Ê	[C_6 \times C_2]$     						&           $Ê	7	$              &  Yes	Abelian

\end{longtable}
\end{center}

\newpage

\begin{thebibliography}{1}

\bibitem{BCFS10}
W.~Bosma, J.J. Cannon, C.~Fieker, and A.~Steel.
\newblock {\em (Eds.) Handbook of Magma Functions,}.
\newblock Edition 2.16, 5017 pages, 2010.

\bibitem{C99}
P.J. Cameron.
\newblock {\em Permutation Groups}.
\newblock Student Texts 45, London Math. Soc., Cambridge University Press,
  1999.

\bibitem{DM96}
J.D. Dixon and B.~Mortimer.
\newblock {\em Permutation Groups}.
\newblock Springer-Verlag New York, 1996.

\bibitem{EP88}
D.~Easdown and C.E. Praeger.
\newblock On minimal faithful permutation representations of finite groups.
\newblock {\em Bull. Austral. Math. Soc.}, 38:207--220, 1988.

\bibitem{J71}
D.L. Johnson.
\newblock Minimal permutation representations of finite groups.
\newblock {\em Amer. J. Math.}, 93(4):857--866, 1971.

\bibitem{S08}
N.~Saunders.
\newblock The minimal degree for a class of finite complex reflection groups.
\newblock {\em J. Algebra}, 323:561--573, 2010.

\bibitem{S10}
N.~Saunders.
\newblock {\em Minimal Faithful Permutation Representations of Finite Groups}.
\newblock Ph.D. Thesis, University of Sydney, 2011.

\bibitem{S14}
N.~Saunders.
\newblock Minimal faithful permutation degrees for irreducible coxeter groups
  and binary polyhedral groups.
\newblock {\em J. Group Theory}, 17:805--832, 2014.

\bibitem{W75}
D.~Wright.
\newblock Degrees of minimal embeddings of some direct products.
\newblock {\em Amer. J. Math.}, 97:897--903, 1975.

\end{thebibliography}
\bibliographystyle{plain}

\end{document}